\documentclass[sigconf]{acmart}
\newif\ifextendedversion
\extendedversiontrue
\newif\ifarxivversion
\arxivversiontrue
\AtBeginDocument{%
  }

\copyrightyear{2026}
\acmYear{2026}
\setcopyright{cc}
\setcctype{by}
\acmConference[KDD '26]{Proceedings of the 32nd ACM SIGKDD Conference on Knowledge Discovery and Data Mining V.2}{August 09--13, 2026}{Jeju Island, Republic of Korea}
\acmBooktitle{Proceedings of the 32nd ACM SIGKDD Conference on Knowledge Discovery and Data Mining V.2 (KDD '26), August 09--13, 2026, Jeju Island, Republic of Korea}
\acmDOI{10.1145/3770855.3818103}
\acmISBN{979-8-4007-2259-2/2026/08}

\usepackage{xcolor}
\usepackage{amsmath}
\usepackage[ruled,noend]{algorithm2e}
\usepackage{multirow}
\usepackage{makecell}
\usepackage{pifont}
\usepackage{booktabs,multirow}
\usepackage{xspace}
\newcommand{\ourmodel}{\textsc{RidgeCut}\xspace}

\newtheorem{prop}{Proposition}
\newtheorem{lemma}{Lemma}

\usepackage{subcaption}
\usepackage{graphicx}
\usepackage{placeins}
\microtypesetup{expansion=false}

\begin{document}

\title{RIDGECUT: Learning Graph Partitioning with Rings and Wedges}

\author{Qize Jiang}
\authornote{Also with Shanghai Institute of Intelligent Electronics \& Systems, Shanghai, China.}
\orcid{0000-0002-1636-8825}
\email{qzjiang21@m.fudan.edu.cn}
\affiliation{%
  \department{College of Computer Science and Artificial Intelligence, Shanghai Key Laboratory of Data Science}
  \institution{Fudan University}
  \city{Shanghai}
  \country{China}
}

\author{Angelo Zangari}
\orcid{0009-0002-1865-7337}
\email{azang@uic.edu}
\affiliation{%
  \institution{University of Illinois Chicago}
  \city{Chicago}
  \state{Illinois}
  \country{United States}
}

\author{Linsey Pang}
\orcid{0000-0002-4784-9795}
\email{panglinsey@gmail.com}
\affiliation{%
  \institution{PayPal Inc.}
  \city{San Jose}
  \state{California}
  \country{United States}
}

\author{Alice Gatti}
\orcid{0000-0001-5692-3996}
\email{alice@safe.ai}
\affiliation{%
  \institution{Center for AI Safety}
  \city{San Francisco}
  \state{California}
  \country{United States}
}

\author{Mahima Aggarwal}
\orcid{0000-0002-3552-9233}
\email{mahima13@gmail.com}
\affiliation{%
  \institution{Qatar Computing Research Institute}
  \institution{Hamad Bin Khalifa University}
  \city{Doha}
  \country{Qatar}
}

\author{Giovanna Vantini}
\orcid{0009-0000-0406-3525}
\email{gvantini@hbku.edu.qa}
\affiliation{%
  \institution{Qatar Computing Research Institute}
  \institution{Hamad Bin Khalifa University}
  \city{Doha}
  \country{Qatar}
}

\author{Xiaosong Ma}
\orcid{0000-0003-1261-2496}
\email{xiaosongma@acm.org}
\affiliation{%
  \department{Computing and Mathematical Sciences (CMS) Division}
  \institution{Mohamed bin Zayed University of Artificial Intelligence (MBZUAI)}
  \city{Abu Dhabi}
  \country{United Arab Emirates}
}

\author{Weiwei Sun}
\authornotemark[1]
\authornote{Corresponding authors.}
\orcid{0000-0001-9483-4599}
\email{wwsun@fudan.edu.cn}
\affiliation{%
  \department{College of Computer Science and Artificial Intelligence, Shanghai Key Laboratory of Data Science}
  \institution{Fudan University}
  \city{Shanghai}
  \country{China}
}

\author{Sourav Medya}
\orcid{0000-0003-0996-2807}
\email{medya@uic.edu}
\affiliation{%
  \institution{University of Illinois Chicago}
  \city{Chicago}
  \state{Illinois}
  \country{United States}
}

\author{Sanjay Chawla}
\authornotemark[2]
\orcid{0000-0002-8102-2572}
\email{schawla@hbku.edu.qa}
\affiliation{%
  \institution{Qatar Computing Research Institute}
  \institution{Hamad Bin Khalifa University}
  \city{Doha}
  \country{Qatar}
}

\renewcommand{\shortauthors}{Qize Jiang et al.}

\begin{abstract}
Reinforcement learning (RL) has shown promise for combinatorial optimization problems on graphs by learning heuristics that generalize across instances. However, effectively incorporating domain knowledge into RL frameworks for graph partitioning remains challenging, as existing approaches typically rely on unconstrained node-level actions that lead to large action spaces and inefficient exploration.
In this paper, we propose \ourmodel, an RL framework that constrains the action space to enforce structure-aware partitioning in the Normalized Cut problem.
Using transportation networks as a motivating example, we introduce a novel concept that leverages domain knowledge about urban road topology---where natural partitions often take the form of concentric rings and radial wedges. By transforming the graph into linear or circular representations, our method enables the use of transformer-based policies and efficient learning via Proximal Policy Optimization. The resulting partitions from \ourmodel are not only aligned with expected spatial layouts but also achieve lower normalized cuts compared to existing methods. Experimental results on synthetic and real-world traffic graphs demonstrate that \ourmodel consistently outperforms existing methods while exhibiting strong inductive generalization across graph sizes. Although motivated by road networks, \ourmodel provides a general mechanism for embedding structural priors into RL frameworks for graph partitioning.

\end{abstract}



\maketitle

\ifarxivversion
\begin{center}
\small This is an extended version of the paper accepted at KDD 2026 and includes appendix.
\end{center}
\fi

\section{Introduction}
Reinforcement Learning (RL) has emerged as a powerful approach for solving complex combinatorial optimization (CO) problems\cite{grinsztajn2023reinforcement, wang2021deep, mazyavkina2021reinforcement}. Two key insights underpin the use of RL in CO: first, that the search space of CO can be effectively encoded as vector embeddings; and second, gradients can be computed even when the objective is a black-box function or non-differentiable. A significant advantage of RL-based approaches is their ability to generalize: once trained, they can solve new instances of CO problems without requiring retraining from scratch  \cite{dong2020deep}.
When addressing real-world problems through combinatorial optimization, 
utilizing insights from the specific domain can lead to 
improved performance compared with standard CO problems. 
However, integrating this domain knowledge into 
RL frameworks poses significant challenges. 
Existing methods often fail to leverage domain knowledge effectively. 
Typically, this knowledge is incorporated by either augmenting states or modifying rewards \cite{DBLP:conf/aaai/PratesALLV19,DBLP:conf/ictai/LemosPAL19,DBLP:conf/nips/KaraliasL20}; 
however, these approaches do not fully incorporate the domain knowledge, and may result in suboptimal performance.
To address this issue, 
we propose a novel approach that employs a constrained action 
space informed by domain knowledge. 
\textit{Domain knowledge and constrained action space. }In this paper, we address a specific problem to demonstrate the effectiveness of constraining the action space in RL for CO problems. Our primary focus is on the Normalized Cut (NC) of a graph, which is suitable to balance the simulating traffic on road networks. Existing RL methods ~\cite{shah2024neurocut} for solving graph partitioning problems typically involve the direct movement of adjacent nodes between partitions. These approaches have full action space, but often require a long or potentially unlimited number of steps to achieve optimal partitions.
In contrast, our method integrates domain knowledge pertaining to ring and wedge partitioning and produces improved partitioning outcomes. Figure \ref{fig:domain-knowledge} (Sec. \ref{sec:our_method}) illustrates an overview of the proposed approach of our work. 
\begin{figure*}[t]
  \centering
  
  \includegraphics[width=1\linewidth]{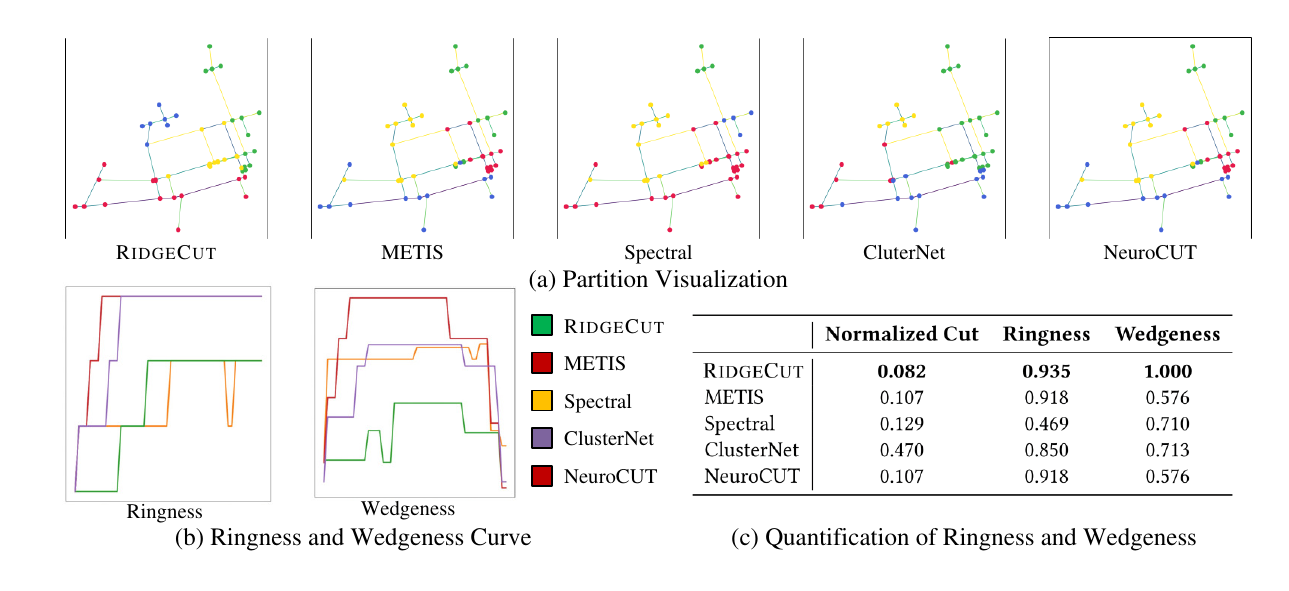}
  \caption{Compared with other methods, \ourmodel has the minimal Normalized Cut, and also achieves the highest Ringness and Wedgeness.
  NeuroCUT is initialized by METIS partition, and fails to find a better one, which causes the same result. 
  }
  \Description{Visual and metric comparison of RidgeCut against graph partitioning baselines on normalized cut, ringness, and wedgeness.}
  
  \label{fig:first-page}
\end{figure*}
\textit{Justification of domain knowledge.} The ring and wedge configuration represents a fundamental structure in modern urban cities. 
Road networks are often organized as concentric
rings of roads centered at a city downtown followed by wedge structures connecting the outer ring.
For example, Beijing's Five Ring includes a large circular expressway 
encircling the city, acting as the outermost "ring" within the urban area, 
while traffic flows from various surrounding suburban areas converge toward the city center.
For microscopic traffic simulation, where the movement of every vehicle is modeled
in a simulator, it often becomes necessary to partition the road network and assign
each partition to a separate simulator in order to reduce the overall simulation time. 
We thus ensure that the partitions respect the natural physical topology of the road network, 
and this understanding serves as a vital domain knowledge when conducting graph partitioning on road networks.
Directly using classical approaches \cite{karypis1997metis}
or modern learning based RL solutions \cite{shah2024neurocut} provide no provision to constrain the generation
of partition shapes justifying the need for a new approach.
\textit{Integration of domain knowledge.} To integrate domain knowledge and thereby constrain the action space, 
our key insight is to convert complex graph structures into 
simpler representations, such as a line or a circle,
and reducing the complexity of the partitioning problem. 
In the ring transformation, nodes are projected onto the $x$-axis according to their radial distance from the center, preserving the node order and partitioning properties. 
Similarly, in the wedge transformation, nodes are projected onto a unit circle, focusing on their angular positions. After transforming the graph, we apply Proximal Policy Optimization (PPO)~\cite{schulman2017proximal} to solve the partitioning problem. 
Figure \ref{fig:first-page} presents performance of different methods.  
In Figure \ref{fig:first-page}(a), a snapshot of the partitions generated by different methods shows that other methods except our proposed method (\ourmodel) tend to mix nodes from different partitions, resulting in high Normalized Cut. Figures \ref{fig:first-page}(b) and \ref{fig:first-page}(c) introduce Ringness and Wedgeness metrics to evaluate how closely a partition aligns with ring and wedge structures. Our method, \ourmodel, achieves the lowest Normalized Cut while maintaining the highest Ringness and Wedgeness scores. Our main contributions are as follows:
\begin{itemize}
    \item \textbf{Novel Direction.} We introduce a new perspective on RL for graph partitioning by embedding domain knowledge directly into the action space. Instead of operating on unconstrained node-level actions, we show that restricting actions to structured partition primitives can substantially simplify exploration while preserving expressiveness. 
    \item \textbf{Novel Method.} We introduce \ourmodel an RL framework that exploits ring- and wedge-based partitioning to transform graph partitioning into a sequence of structured decisions. By converting graphs into linear and circular representations, \ourmodel enables the use of transformer-based approaches and efficiently optimizes the Normalized Cut objective under a constrained action space.
    \item \textbf{Experiments. }Through extensive experiments on synthetic graphs and real-world traffic networks, we demonstrate that \ourmodel consistently achieves lower normalized cuts, produces partitions aligned with meaningful spatial structures, and exhibits strong inductive generalization across graph sizes without retraining.
\end{itemize}

\section{Related Work} 
\label{sec:related-work}
\textbf{Graph Partitioning. }Graph partitioning~\cite{buluc:partitioning} is widely used in graph-related applications, especially for enabling parallel or distributed graph processing. 
Partitioning a graph into $k$ blocks of equal size while minimizing cuts is NP-complete~\cite{HR73}. Exact methods focus on bipartitioning~\cite{hager:exact} or few partitions ($k \leq 4$)~\cite{ferreira:partitioning}, while approximate algorithms include spectral partitioning~\cite{DBLP:conf/nips/NgJW01,donath:lowerbound} and graph-growing techniques~\cite{George:growing}. More powerful methods involve iterative refinement, such as node-swapping for bipartitioning~\cite{KL70}, extendable to $k$-way local search~\cite{Karypis:kway}. Other approaches include the bubble framework~\cite{diekmann:shape} and diffusion-based methods~\cite{meyerhenke:diffusion, pellegrini:diffusion}. State-of-the-art techniques rely on multilevel partitioning~\cite{karypis:multilevel}, which coarsen the graph and refine the partition iteratively.
\begin{figure*}[ht]
  \centering
  
  \includegraphics[width=1\linewidth]{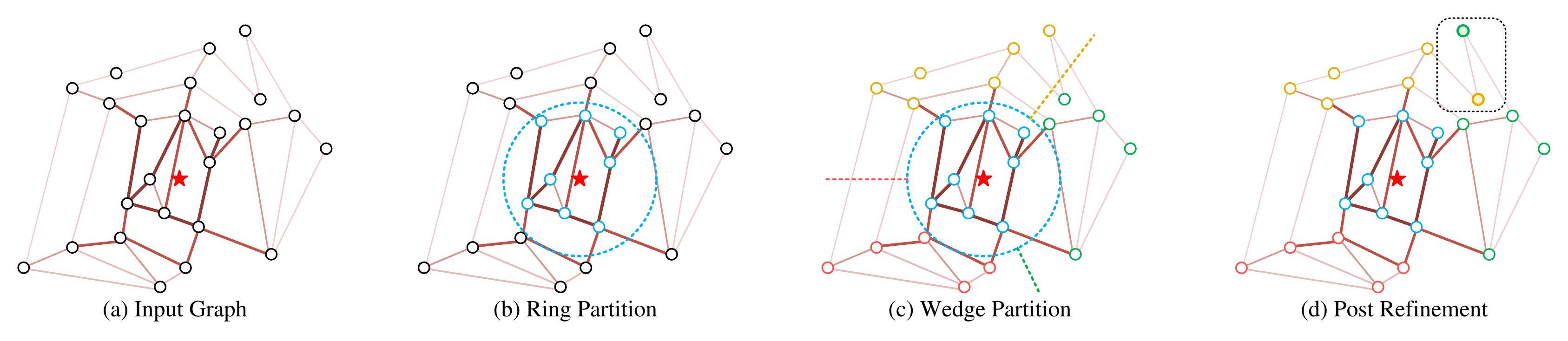}
\caption{Graph partitioning with Ring and Wedge to minimize the Normalized Cut. We first do ring partitions as (b), to choose different radii to partition the graph into rings. Then for the outermost ring, we do partitions based on different angles as (c). Finally, we apply post-refinement to improve the final partition performance as (d). 
  }
  \Description{Illustration of ring partitioning, wedge partitioning, and post-refinement on a graph.}
  \label{fig:ring-wedge-sample}
\end{figure*}
One of the most well-known tools is METIS~\cite{Metis, karypis:multilevel}, 
which uses multilevel recursive bisection and $k$-way algorithms, 
with parallel support via ParMetis~\cite{ParMetis} 
and hypergraph support via hMetis \cite{DBLP:conf/dac/KarypisAKS97,DBLP:journals/tvlsi/KarypisAKS99}. 
Other tools include Scotch~\cite{scotch, pellegrini:diffusion} and KaHIP~\cite{sanders:engineering} use various advanced techniques. 
However, these methods fail to give optimal partitions with normalized cuts, 
and are also unable to integrate extra domain knowledge.
\textbf{ML-based Graph Partitioning Algorithms. }Recent work explores ML methods for graph partitioning, particularly using GNN \cite{DBLP:conf/iclr/KipfW17,DBLP:conf/iclr/VelickovicCCRLB18,DBLP:conf/iclr/Brody0Y22}. GNN aggregate node and edge features via message passing. In~\cite{dlspectral2022}, a spectral method is proposed where one GNN approximates eigenvectors of the graph Laplacian, which are then used by another GNN for partitioning. The RL-based method in~\cite{rlgp2022} refines partitions in a multilevel scheme. NeuroCUT~\cite{shah2024neurocut} introduces a reinforcement learning framework generalizing across partitioning objectives via GNN. It demonstrates flexibility for different objectives and unseen partition numbers. ClusterNet~\cite{wilder2019end} integrates graph learning and optimization with a differentiable k-means clustering layer, simplifying tasks like community detection and facility location. 
However, neither of these methods handles weighted graphs, making them unsuitable in our scenarios.

\section{Problem Description}
Let $G = (V,E,W,o)$ be a weighted planar graph, with vertex set $V$, edge set $E$, edge weights $W$, and a predefined center $o$. A $k$-way partition $P$ of $G$ is defined as a partition $\{p_1, ..., p_k\}$ of $V$, where $\bigcup_{i=1}^kp_i=V$ and $\forall i \neq j, p_i \cap p_j = \emptyset$.
 For each partition $p_i$, we define
\begin{align}
  Cut(G, p_i) = \sum_{u \in p_i \otimes v \in p_j}W(e_{u,v}) 
 \end{align} 
\begin{align}
Volume(G, p_i) =\sum_{u, v \in p_i}W(e_{u,v}) + Cut(G, p_i)
\end{align}
where $\otimes$ represents the XOR operator.
The \textit{Normalized Cut} of a partition $P$ on graph $G$ is then defined as
\begin{equation}\label{eq: ncut_definition}
    NC(G, P) = \max_{i\in\{1..k\}}\frac{Cut(G,p_i)}{Volume(G,p_i)}.
\end{equation}
To solve this, we utilize domain knowledge that 
ring and wedge partitions are effective for road networks, so we restrict our attention to partitions with specific structures, 
where each partition is either ring-shaped or wedge-shaped. 
We also allow for ``fuzzy" rings and wedges, where a small number of nodes 
are swapped to adjacent partitions. This relaxation helps achieve partitions 
with a smaller normalized cut, particularly for graphs derived from 
real-world applications. Our partitioning strategy follows a three-step process: first, we perform a ring partition on the graph, then we apply a wedge partition to the outermost rings, and finally, we refine the partitions to further reduce the normalized cut. Figure~\ref{fig:ring-wedge-sample} illustrates these three steps.
\noindent\textbf{Ring Partition:} A Ring Partition of the graph $G$ with respect to the center $o$, denoted by $P^r$, divides $G$ into $k_r$ distinct concentric rings. 
Define the radii as $0 = r_0 \leq r_1 \leq r_2 \leq \dots \leq r_{k_r-1} < r_{k_r}$. These radii partition $G$ into $k_r$ rings, where the $i$-th ring, denoted as $p^r_i$, contains all nodes with a distance to the center $o$ between $r_{i-1}$ and $r_i$. 
\noindent\textbf{Wedge Partition:} 
A Wedge Partition, denoted as $P^w$, divides the outermost ring $p^r_{k_r}$ into multiple wedge-shaped sections. The partitioning angles are given by $0 \leq a_1 \leq a_2 \leq \dots \leq a_{k_w} < 2\pi$. 
These angles split $p^r_{k_r}$ into $k_w$ wedge parts, where the $i$-th wedge, $p^w_i$, contains the nodes whose polar angles are between $[a_i, a_{i+1})$, except for the wedge $p^w_{k_w}$, which contains nodes whose angles fall within either $[0, a_1)$ or $[a_{k_w}, 2\pi)$.
This type of partition divides the graph into $k_r - 1$ inner rings and $k_w$ wedges on the outermost ring (see Figure \ref{fig:ring-wedge-sample}). Specifically, if $k_r = 1$, the entire graph is partitioned solely by wedges and, conversely, if $k_w = 1$ the graph is partitioned solely by rings.
For simplicity, when a graph $G$ is partitioned by a Ring-Wedge Partition with $k_r$ and $k_w$, we define $k = k_r + k_w - 1$, with $p_k=p^r_k$ when $k < k_r$, and $p_k=p^w_{k-k_r+1}$ when $k>=k_r$. 
And we define the total partition strategy as $P=\{p_1, ..., p_k\}$.
We also propose Ringness and Wedgeness metrics to evaluate whether a partition is close to a ring or wedge shape and provide theoretical analyses for ring and wedge subsets.
\ifextendedversion
\textit{Detailed definitions of Ringness and Wedgeness are provided in Appendix \ref{app:xness}.}
\fi

\subsection{Theoretical Analyses for Ring and Wedge Partition }
Our main theoretical results show that the normalized cut defined in \eqref{eq: ncut_definition} satisfies bounds on the Cheeger constant similar to the classical setting \cite{chung1997spectral} in the case of unweighted spider web graphs. These bounds give a theoretical justification of the normalized cut definition \eqref{eq: ncut_definition} and the ring-wedge shaped partition.
\ifextendedversion
Details and proofs are provided in Appendix \ref{app:bound}.
\fi
\\
\noindent\textbf{Definition:} \textit{Let $G_{n,r}$ be an unweighted spider web graph with $r$ rings and $n$ points in each ring, and $k$ be an integer. Define the ring and wedge Cheeger constants as:}
\begin{align}
\phi_{n,r}(k) &\!=\! \min_{\substack{P=V_1\cup\cdots\cup V_k \\ \text{wedge partition}}} NC(G_{n,r}, P) 
\end{align}
\begin{align}
\psi_{n,r}(k) &\!=\! \min_{\substack{P=V_1\cup\cdots\cup V_k \\ \text{ring partition}}} NC(G_{n,r}, P).
\end{align}
\begin{prop}\label{prop: 1}
Let $G_{n,r}$ be a spider web graph with $r$ rings and $N$ nodes in each ring. Let $\lambda_k^C$ and $\lambda_k^P$ be the eignevalues of the circle and path graphs with $n$ and $r$ vertices respectively. Then
\begin{align}
\phi_{n,r}(k) &\leq \frac{2r}{2r-1}\sqrt{2\lambda_k^C}, &\quad 2 \leq k \leq n \\
\psi_{n,r}(k) &\leq \sqrt{2\lambda_k^P}, &\quad 2 \leq k \leq r
\end{align}
 \end{prop}

\section{Our Method: \ourmodel}
\label{sec:our_method}
We begin by introducing the settings of reinforcement learning environment and provide a comprehensive overview of the agent's function and its interactions with the environment to achieve the final partitioning. For simplicity, we pre-define the ring partition number as $k_r$ and wedge partition number as $k_w$. When $k$-partitioning  a graph, we enumerate all possible ring partition numbers, then select the one that minimizes the Normalized Cut.
\subsection{The RL Framework of \ourmodel}
 We propose an RL-based framework called \ourmodel (\underline{RI}ng and We\underline{DGE} Informed \underline{CUT}) to solve normalized cut. Figure~\ref{fig:domain-knowledge} provides an overview of the proposed framework. Existing RL-based approaches for graph partitioning~\cite{shah2024neurocut} typically operate by iteratively relocating individual nodes between partitions. While such methods expose a fully expressive action space, they often rely on long action sequences to reach high-quality partitions, which complicates learning and limits efficiency. In contrast, our approach incorporates domain knowledge through ring- and wedge-based partitioning primitives, effectively constraining the action space to structurally meaningful operations. This design enables more direct exploration of high-quality partitions and simplifies the learning process by aligning actions with the underlying graph structure.
Specifically, we use Proximal Policy Optimization (PPO) \cite{schulman2017proximal} to train the model with non-differential optimizing targets. PPO is a widely-used RL algorithm that optimizes the policy by minimizing a clipped surrogate objective, ensuring limited deviation from the old policy $\pi_\text{old}$. The PPO objective maximizes the following:
\[
\mathbb{E}_{t} \left[ \min(r_t(\theta) A_t, \text{clip}(r_t(\theta), 1-\epsilon, 1+\epsilon) A_t) \right],
\] 
where $r_t(\theta) = \frac{\pi_{\theta}(a_t|s_t)}{\pi_{\theta_\text{old}}(a_t|s_t)}$ and $A_t$ is the advantage. In \ourmodel, the observation space, action space and reward function are defined below.
The final goal is to maximize the reward through interactions with the environment described above.
\textbf{Observation Space.} The observation space $S$ contains the full graph $G$, the expected ring number $k_r$, wedge number $k_w$, and the current partition $P$, denoted by $S = \{ G, k_r, k_w, P \}$.
\textbf{Action Space.} 
The agent needs to decide the next partition as action. If it is a Ring Partition, the action is the \emph{radius} of next ring, if it is a Wedge Partition the action is the partition \emph{angle} of the wedge.
\[
A=\left\{
\begin{array}{ll}
r & \text{if \; currently \; expects \; a \; ring \; partition} \\
a & \text{if \; currently \; expects \; a \; wedge \; partition}
\end{array}
\right.
\]
\textbf{Reward Function.} When the partitioning process is incomplete, we assign a reward of 0. Conversely, when the partitioning process is complete, the reward is thereby formulated as the negative of the Normalized Cut, reflecting our objective to minimize this metric, i.e. $r = -NC(G, P)$.
\begin{figure}[tb]
\centering
\includegraphics[width=0.45\textwidth]{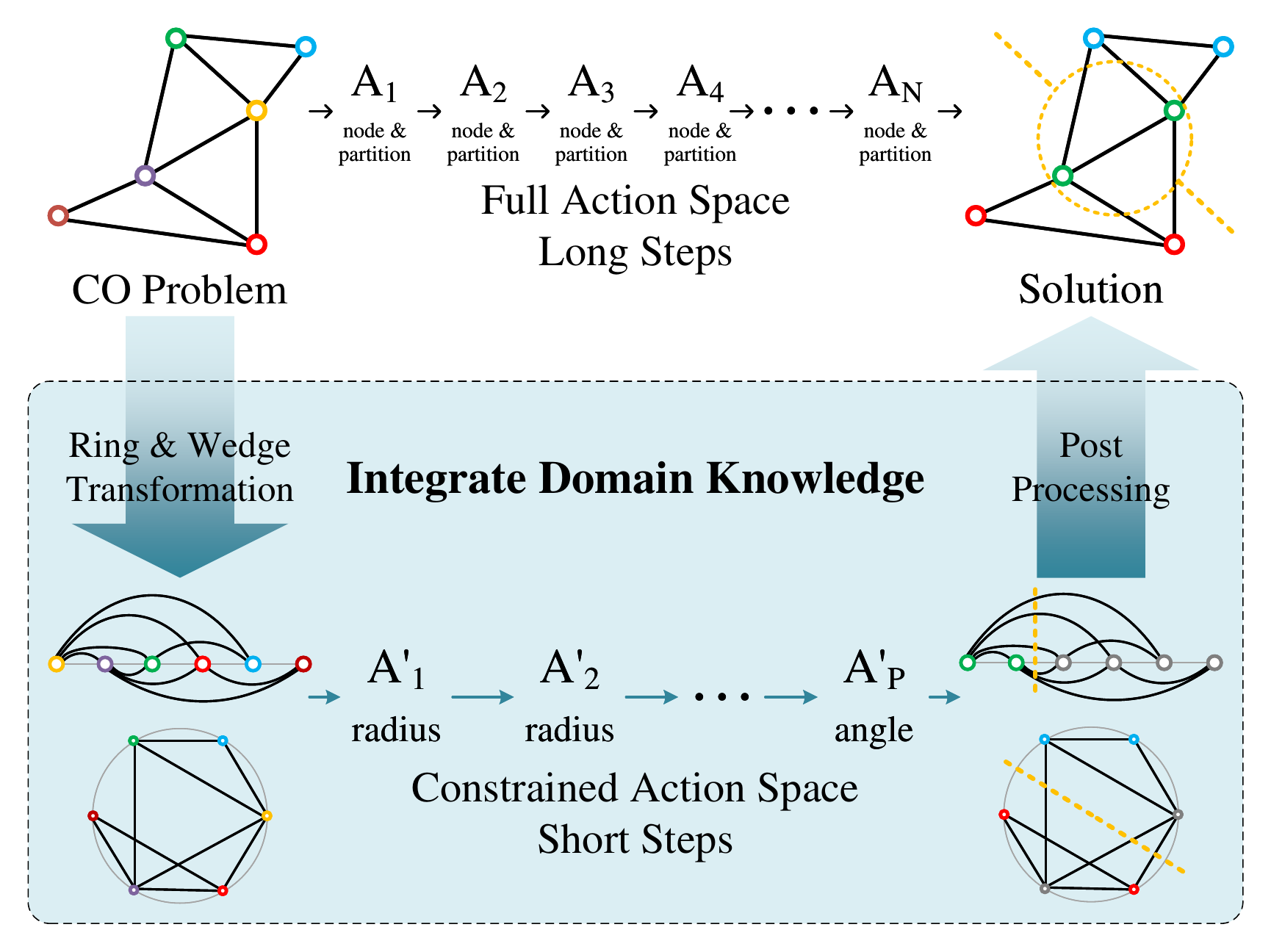}
\caption{We design a solution (\ourmodel) by integrating domain knowledge and learning in the constrained action space (blue box) while existing RL frameworks for combinatorial optimization (CO) problems follow the path above.}
\Description{Workflow diagram contrasting standard RL for combinatorial optimization with RidgeCut's domain-knowledge-constrained action space.}
\label{fig:domain-knowledge}
\end{figure}
\subsection{Transforming the Graph into Rings and Wedges}
\label{subsec:graphtrans}
Most of the previous learning-based graph partitioning methods involve a combination of GNNs and RL \cite{shah2024neurocut}. However, GNNs face limitations as they only aggregate the local structure of the graph, necessitating an initial partition followed by fine-tuning. In contrast, \ourmodel produces ring and wedge partition results directly. To generate a good partition, we aim for \ourmodel to learn the global view of the graph. To achieve this we use a Transformer-based architecture. Recently, Transformers have achieved remarkable success across various domains, utilizing Multi-Head Attention to facilitate global information exchange, thus demonstrating superior performance in diverse tasks. 
However, Transformers typically require sequential input, rendering them 
ill-suited for graph representations. To address this issue, we employ two 
transformations, Ring Transformation and Wedge Transformation, on the graph.  
\textit{These new representations maintain equivalence with the original graph to generate Ring Partitioning or Wedge Partitioning, while making the data sequential to enable the transformer-based model.}
\begin{figure*}[t]
  \centering
  \includegraphics[width=\linewidth]{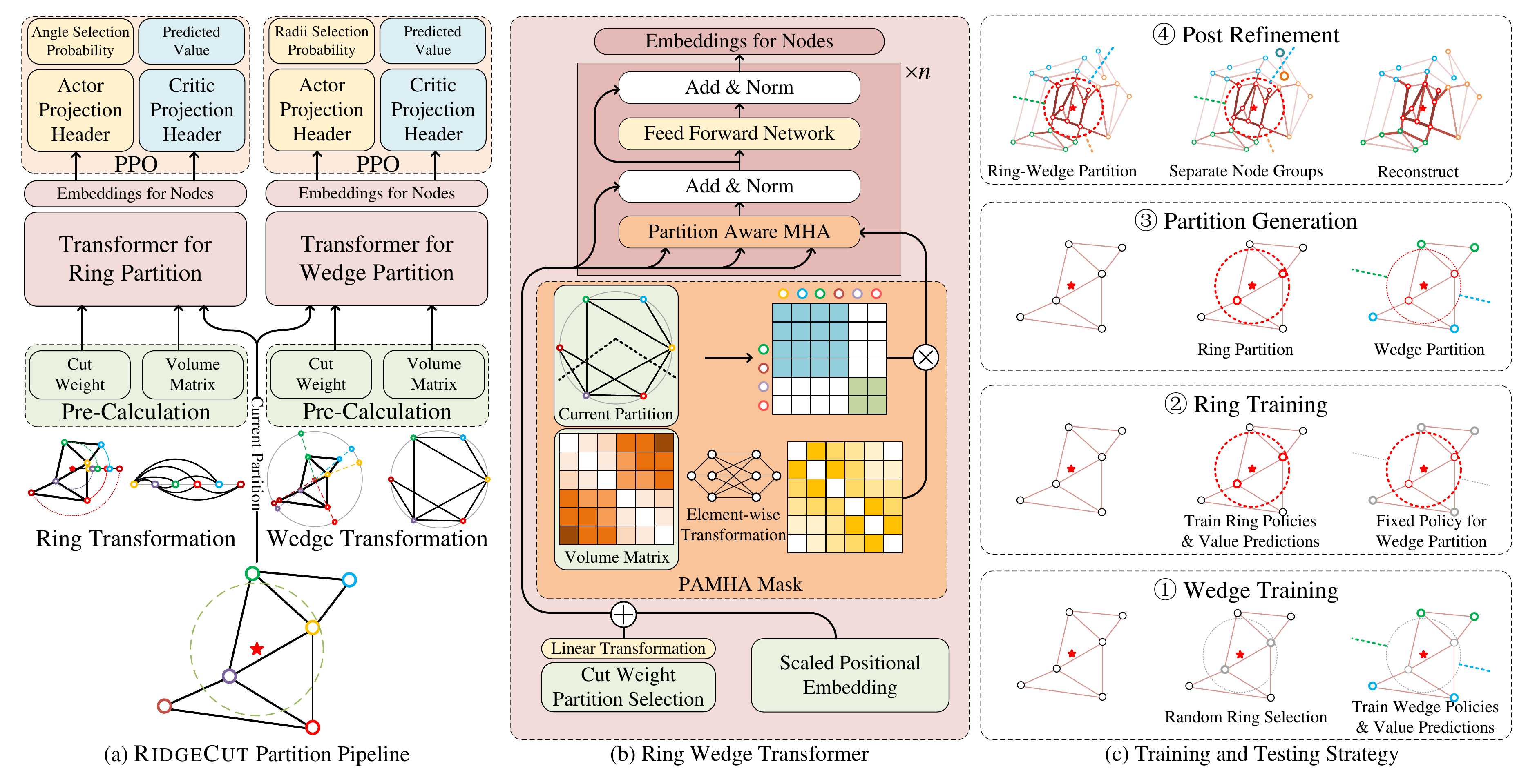}
  \caption{The framework of \ourmodel. (a) Our method first applies Ring and Wedge Transformations, followed by pre-calculation to obtain Cut Weight and Volume Matrix. The processed data generates node embeddings for action probabilities and predicted values via actor and critic projection headers. Modules for ring and wedge partition share structures but differ in weights. (b) Detailed structure of Ring Wedge Transformer, using cut weights with positional embeddings as input, followed by transformer layers. 
  Partition Aware MHA Mask is used, which considers Volume Matrix and 
	  Current Partition to generate appropriate masks.
	 (c) Pipeline from training to testing. Initially, the wedge partition strategy is trained with a random approach for the ring partition, then the wedge part is fixed while training the ring part, excluding its critic projection header. In testing, \ourmodel sequentially determines ring radius and wedge angle, and refines the final partition using post-refinement. 
	  }
\Description{RidgeCut framework diagram showing graph transformations, pre-calculation, transformer-based policy modules, and the training and testing pipeline.}
\label{fig:model}
\end{figure*}
\subsubsection{Ring Transformation.}
First, note that the ring partition remains unchanged when rotating the graph around 
a chosen center $o$, we can project each node on the $x$-axis. More specifically, 
if a node has polar coordinates $(r, \theta)$, the projection will map it 
to the coordinate $(r, 0)$. Figure \ref{fig:model} (a)
\ifextendedversion
and Figure \ref{fig:transform} (b)
\fi
illustrate this projection onto the line.
When the order of nodes on the line remains unchanged, 
we can adjust the radius of any point, and the partition results in the new graph will be identical to those in the original graph. By applying this 
observation, we can transform a standard graph into a simplified version where 
every node is represented by the coordinates $(X, 0)$, where $X$ is the radius 
order of the node among all nodes.
This transformation does not alter the order of the nodes or the existing partitions. 
The final result of this transformation is shown on the left in Figure \ref{fig:model}(a)
\ifextendedversion
and Figure \ref{fig:transform}(b)
\fi
.
\subsubsection{Wedge Transformation.}
Similar to Ring Transformation, we observe that during wedge partition, 
the radius of nodes does not influence the outcome; only the angle of the 
nodes is taken into account. We project all nodes onto a unit circle centered 
at point $o$. Thus, if the polar coordinates of a node are \( (r, X) \), 
its corresponding projection will have coordinates \( (1, X) \). 
After implementing the projection, we also modify the angles of the nodes. 
If a node's angle order is \( X \) among \( N \) nodes, its new polar 
coordinate will be adjusted to \( (1, \frac{2\pi X}{N}) \). 
Figure \ref{fig:model} (a)
\ifextendedversion
and Figure \ref{fig:transform}(c)
\fi
demonstrate the transformation.
\subsubsection{Action Space for RL}With these transformations, nodes of the graph are positioned either along a line or on a circle, allowing us to treat the graph as a sequential input. 
Additionally, actions that split nodes \( i \) and \( i + 1 \) into two 
partitions yield identical final partition results. 
Consequently, we can convert the continuous action space into discrete 
actions to mitigate learning difficulties. A new action \( A_i \) indicates 
the partitioning of nodes \( i \) and \( i + 1 \) into two separate partitions.
\subsection{Partition Pipeline of \ourmodel}
\label{subsec:pipeline}
The graph partition pipeline of \ourmodel is illustrated in Figure \ref{fig:model} (a). It sequentially determines partitions through Ring and Wedge Transformations, iteratively predicting the next ring radius or wedge angle until the target partition count is achieved. The model consists of two components for ring and wedge partitions, sharing similar structures but distinct weights,
and it contains \textbf{four sub-modules.} 
\noindent
\textbf{(i) Transformation:} The appropriate transformation (ring or wedge) is applied based on current requirements. 
\noindent
\textbf{(ii) Pre-Calculation:} Essential computations on the transformed graph include: (1) Cut Weight \(C_i\): Sum of edge weights crossing between nodes \(i\) and \(i+1\). (2) Volume Matrix \(V_{i,j}\): Total weight of edges covered between nodes \(i\) and \(j\) (where \(i < j\)). 
\noindent
\textbf{ (iii) Ring Wedge Transformer Module:} The Transformer processes node embeddings derived from the pre-calculation phase and the current partition status, as depicted in Figure \ref{fig:model} (b).
\ourmodel takes as input the Cut Weight, Volume Matrix, and Current Partition from the Pre-Calculation module. Each node’s selection in the Current Partition is encoded as a 0-1 array, combined with Cut Weight, and passed through a linear layer, followed by positional embeddings. These representations are processed by a stack of Transformer blocks to produce final node embeddings.
To enhance structural awareness, we replace the Multi-Head Attention with Partition-Aware Multi-Head Attention (PAMHA). PAMHA integrates the Volume Matrix and Current Partition into its attention mask. The mask selectively emphasizes relevant node pairs based on their partition status and graph locality. For example, in a circular graph, changes to one segment may not influence distant nodes—this insight is embedded into the attention mechanism by restricting attention to the effective node range. The final node embeddings generated by \ourmodel are then passed to the PPO module for policy optimization.
\ifextendedversion
\textit{Additional details on this module are provided in Appendix \ref{app:transformer}.}
\fi
\noindent
\textbf{(iv) PPO Header:} After receiving node embeddings, the PPO header extracts action probabilities and critic values. The actor projection header maps hidden size \(h\) to dimension 1, followed by a Softmax layer for action probabilities. Value prediction uses Self-Attention average pooling on node embeddings and projects from \(h\) to 1. 
The model will perform one action with one forward,
and is employed to execute actions recursively until the graph is fully partitioned.
Also, during the Ring Partition phase, a dynamic programming algorithm calculates the optimal partition when the maximum radius and total ring count are fixed, with a complexity of \(O(n^2k)\). 
Thus, \ourmodel determines the maximum radius for all ring partitions only once.
\ifextendedversion
The pseudo-code is available in Appendix \ref{app:pseudo}.
\fi
\textbf{Training and Testing.} We employ specialized training and testing strategies to enhance policy learning and improve partition results. Both the training and testing processes are divided into two stages. A visualization of these four stages is depicted in Figure \ref{fig:model} (c).
\ifextendedversion
\textit{Training and testing details are provided in Appendix \ref{app:traintest}.}
\fi

\section{Experimental Results}
\label{sec:exp}
To demonstrate the performance of our proposed \ourmodel, 
which integrates domain knowledge about ring and wedge partitions, we evaluate our model along with other baselines using both synthetic and real-world graphs.
We first introduce the dataset details, 
then give the competitors in graph partitioning, 
and finally show the overall performance and ablation studies results. 
The code is available at \url{https://github.com/zyr17/RIDGECUT}.
\subsection{Set up}
\subsubsection{Graph Datasets}
\label{subsec:dataset}
To evaluate methods precisely, we construct three graph datasets. The datasets are defined as follows:
\noindent\textbf{Predefined-weight Graph:}
This dataset resembles a spider web and consists of $N$ concentric circles, each having $M$ equally spaced nodes. 
The radii of circles are from 1 to $N$. 
We build unweighted spider web graphs with randomly chosen circles and number of nodes, 
then apply a random valid ring-wedge partition,
specifying both the number of rings and wedges.
We then assign lower weights to edges that cross different partitions and higher weights to edges within the same partition (intra-partition edges). 
A synthetic example is given in Figure \ref{fig:both_graphs} (a).
\ifextendedversion
\textit{Details on ranges of nodes, circles, weights etc, for generating the graphs are in Appendix \ref{app:stat}}. 
\fi
\begin{figure}
    \centering
    \includegraphics[scale=0.06]{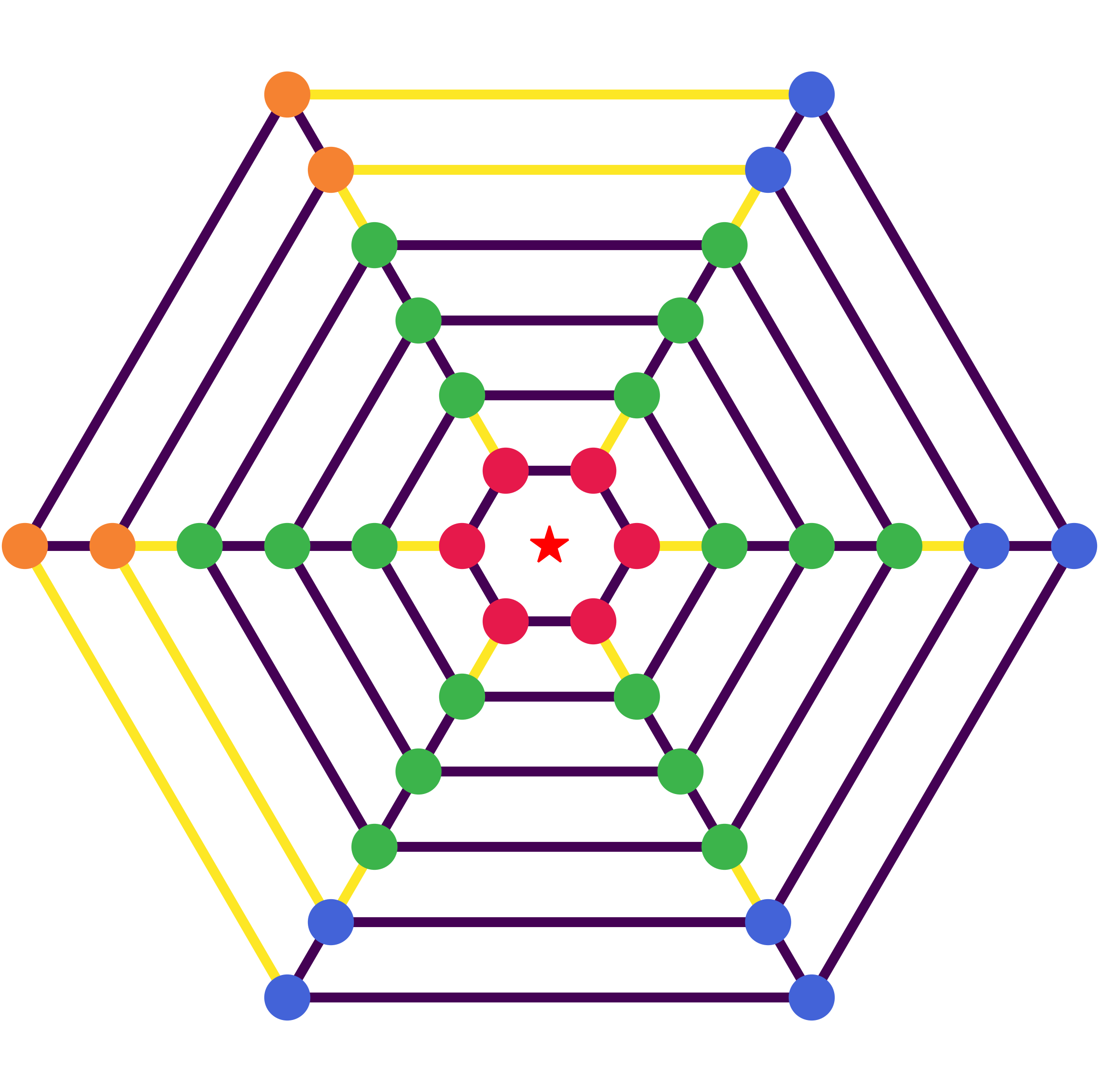}
    \hspace{1cm}\includegraphics[scale=0.02]{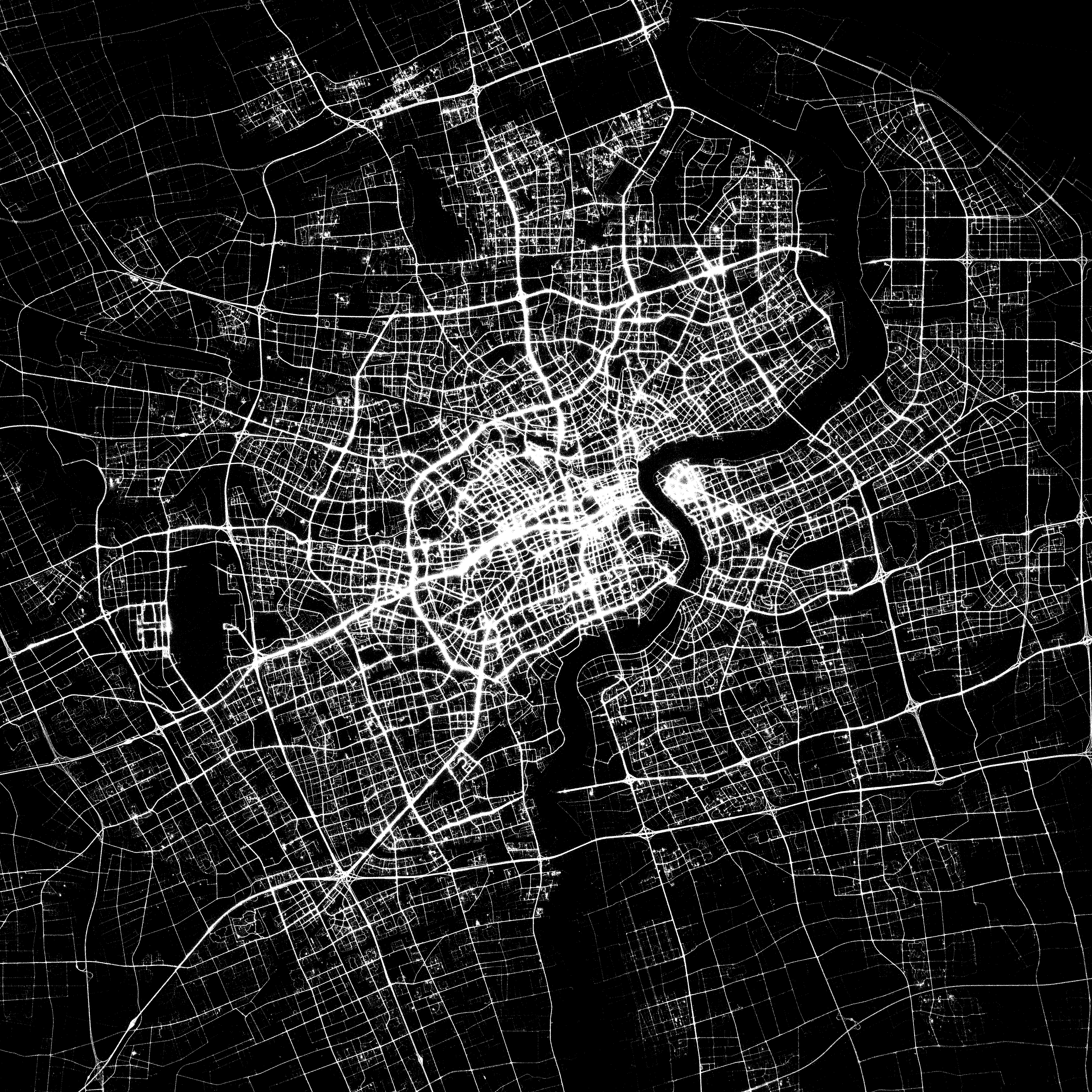}
    \caption{Visualization of graphs used in experiments. (Left) Synthetic graph: 6 circles and 6 nodes on circle. 
        Yellow edges are lower weights, others are higher weights. 
        Current partition 2 rings and 2 wedges. (Right) City Traffic Map: Overview of a real traffic map. We randomly sample connected sub-graphs for training and testing. }
    \Description{Examples of a synthetic ring-wedge graph and a sampled real city traffic graph used in the experiments.}
    \label{fig:both_graphs}
\end{figure}
\noindent\textbf{Random-weight Graph:} 
The graph structure is the same as above, but edge weights are assigned randomly in a given range.
In the random-weight graphs, models should find the best partition without prior knowledge.
\ifextendedversion
The statistics of our training and test synthetic graphs are shown in Table \ref{tab:stas}.
\fi
\noindent\textbf{Real City Traffic Graph:}
For real-world data, we utilize sub-graphs randomly extracted from a comprehensive city traffic map (Figure \ref{fig:both_graphs} (b)). 
The extracted sub-graph is always connected.
For edge weights, we collect traffic data of the city during a specific time range to assess our method's ability to handle real traffic effectively.
\ifextendedversion
The statistical information of Real City Traffic Graph can be found in Table \ref{tab:spider_graph_specs}.
\fi
The real-world traffic dataset used in our experiments was provided by a third-party data partner under a data-use agreement. Due to licensing restrictions, the raw data cannot be publicly released. We report aggregated statistics and experimental results only and do not disclose proprietary raw records.
\begin{table}[t]
\footnotesize
\centering
\caption{Performance on Predefined-weight, Random-weight, and City Traffic Graphs by Normalized Cut. 
Lower values indicate better performance. 
}
\label{tab:overall_transformed}
\resizebox{\columnwidth}{!}{
\begin{tabular}{@{}l*{6}{c}@{}}
\toprule
\multirow{4}{*}{Method} & \multicolumn{6}{c}{4 Partition} \\
\cmidrule(l){2-7}
& \multicolumn{2}{c}{Predefined-weight} & \multicolumn{2}{c}{Random-weight} & \multicolumn{2}{c}{City Traffic} \\
\cmidrule(lr){2-3} \cmidrule(lr){4-5} \cmidrule(l){6-7}
& 50 & 100 & 50 & 100 & 50 & 100 \\
\midrule
Metis & .069 & {.036} & {.065} & .033 & {.245} & .162 \\
Spec. Clust. & .065 & {.036} & .079 & .041 & .384 & .218 \\
Bruteforce & .070 & .036 & {.070} & .036 & .361 & .237 \\
Random & .076 & .040 & .080 & .041 & {.209} & {.095} \\
NeuroCut & \underline{.059} & \underline{.032} & \underline{.064} & \underline{.033} & \underline{{.192}} & \underline{{.078}} \\
ClusterNet & .078 & .043 & .093 & .043 & {.507} & .261 \\
\ourmodel & \textbf{{.042}} & \textbf{.021} & \textbf{{.057}} & \textbf{{.029}} & {\textbf{.174}} & \textbf{{.060}} \\
\hline\hline
\multirow{4}{*}{Method} & \multicolumn{6}{c}{6 Partition} \\
\cmidrule(l){2-7}
& \multicolumn{2}{c}{Predefined-weight} & \multicolumn{2}{c}{Random-weight} & \multicolumn{2}{c}{City Traffic} \\
\cmidrule(lr){2-3} \cmidrule(lr){4-5} \cmidrule(l){6-7}
& 50 & 100 & 50 & 100 & 50 & 100 \\
\midrule
Metis & {.097} & {.053} & .094 & \underline{.049} & {.383} & {.304} \\
Spec. Clust. & .099 & {.053} & .101 & .053 & .652 & .843 \\
Bruteforce & .106 & .054 & .107 & .054 & .615 & .457 \\
Random & .144 & .074 & .142 & .072 & .512 & .341 \\
NeuroCut & \underline{.086} & \underline{.046} & \underline{.093} & \underline{.049} & \underline{.348} & \underline{.226} \\
ClusterNet & .106 & .070 & .120 & {.083} & .837 & .747 \\
\ourmodel & \textbf{{.062}} & \textbf{{.032}} & {\textbf{.081}} & {\textbf{.041}} & {\textbf{.317}} & {\textbf{.182}} \\
\bottomrule
\end{tabular}
}
\end{table}
\subsubsection{Models and Compared Methods}
We compare our proposed method with the following baselines and methods. 
For traditional approaches, we select:
\textbf{METIS}, a solver that is used to partition graphs with balanced size. 
\textbf{Spectral Clustering} uses eigenvectors and k-means to perform graph partitioning.
We also propose two baselines for ring and wedge partitions:
\textbf{Bruteforce} method to enumerate possible ring and wedge partitions. 
\textbf{Random} to randomly generate 10,000 partitions and choose the best-performing one as the result.
For Reinforcement Learning based graph partitioning methods, we select two state-of-the-art methods, 
\textbf{ClusterNet} and \textbf{NeuroCUT}, which are introduced in Section \ref{sec:related-work}.
Finally, we compare the above methods with our proposed \ourmodel.
\ifextendedversion
We also compare these methods with the variants of \ourmodel, and the results are
in Appendix \ref{app:ablation} due to the space limitation.
\fi
\begin{table}[t]
\footnotesize
\centering
\caption{
Results on performance for Inductivity on node sizes for \ourmodel. We train on 100 nodes and test on 50 or 200 nodes. The baselines are not inductive and non-learning based, and still \ourmodel outperforms them.
}
\label{tab:zeroshot_transformed}
\resizebox{\columnwidth}{!}{
\begin{tabular}{@{}l*{6}{c}@{}}
\toprule
\multirow{3}{*}{Method} & \multicolumn{6}{c}{4 Partition} \\
\cmidrule(l){2-7}
& \multicolumn{2}{c}{Predefined-weight} & \multicolumn{2}{c}{Random-weight} & \multicolumn{2}{c}{City Traffic} \\
\cmidrule(lr){2-3} \cmidrule(lr){4-5} \cmidrule(l){6-7}
& 50 & 200 & 50 & 200 & 50 & 200 \\
\midrule
METIS & \underline{.069} & .019 & \underline{.065} & \underline{.016} & .245 & \underline{.048} \\
Bruteforce & .070 & \underline{.018} & .070 & .018 & .361 & .175 \\
Random & .076 & .021 & .080 & .021 & \underline{.209} & .512 \\
\ourmodel & \textbf{.052} & \textbf{.013} & \textbf{.061} & \textbf{.016} & \textbf{.158} & \textbf{.023} \\
\hline\hline
\multirow{3}{*}{Method} & \multicolumn{6}{c}{6 Partition} \\
\cmidrule(l){2-7}
& \multicolumn{2}{c}{Predefined-weight} & \multicolumn{2}{c}{Random-weight} & \multicolumn{2}{c}{City Traffic} \\
\cmidrule(lr){2-3} \cmidrule(lr){4-5} \cmidrule(l){6-7}
& 50 & 200 & 50 & 200 & 50 & 200 \\
\midrule
METIS & \underline{.097} & \underline{.027} & \underline{.094} & \underline{.024} & \underline{.383} & \underline{.086} \\
Bruteforce & .106 & .028 & .107 & .028 & .615 & .311 \\
Random & .144 & .037 & .142 & .037 & .512 & .212 \\
\ourmodel & \textbf{.066} & \textbf{.017} & \textbf{.087} & \textbf{.022} & \textbf{.323} & \textbf{.085} \\
\bottomrule
\end{tabular}

}
\end{table}
{
\begin{table}[t]
\footnotesize
\centering
\caption{Results on Ringness and Wedgeness evaluation }
\label{tab:ringness-wedgeness}
\resizebox{\columnwidth}{!}{
\begin{tabular}{cccccc}
\toprule
          & METIS & Spe.Clu. & NeuroCUT & ClusterNet & \ourmodel   \\ \hline
Ringness  & 0.871 & 0.776       & 0.840    & 0.854      & \textbf{0.929} \\
Wedgeness & 0.587 & 0.810       & 0.621    & 0.820      & \textbf{0.876} \\
\bottomrule
\end{tabular}
}
\end{table}
}
\subsection{Evaluation of Performance}
We show the overall performance in Table \ref{tab:overall_transformed}. We tested our model in three different types of datasets (Section \ref{subsec:dataset}), with 4 or 6 partition numbers.
\ifextendedversion
Dataset statistics are summarized in Table \ref{tab:stas}.
\fi
The number of graphs used for training is $400,000$.
We test the performance of different methods on 100 randomly generated graphs and report the average performance.
We find that our method consistently performs best on all datasets, showing the advantage of the reduced ring-wedge shaped action space.
Although Metis and Spectral Clustering can give graph partitions with any shape, they still cannot reach better performance compared with our proposed method, because it is hard to find best results in such huge action space.
Two basic methods, Bruteforce and Random, 
although they follow the ring and wedge partition,
their performances are always worst compared with other methods, 
because they do not consider the differences of edge weights, and only do random partitioning.
\subsection{Inductivity on Node Sizes}
\begin{table}[t]
\footnotesize
    \centering
    \caption{Inference time comparison of different methods}
    \label{tab:inference}
    \resizebox{.95\columnwidth}{!}{
    \begin{tabular}{c|ccccc}
        \toprule
        Size & METIS & Spectral & NeuroCUT & ClusterNet & \ourmodel \\
        \hline
        50 & 0.073 & 0.011 & 0.845 & 0.140 & 0.050 \\
        100 & 0.062 & 0.034 & 1.596 & 0.145 & 0.073 \\
        200 & 0.077 & 0.507 & 2.512 & 0.140 & 0.261 \\
        \bottomrule
    \end{tabular}
    }
\end{table}
We train the model on three types of graphs, each containing a fixed number 
of nodes (N=100) within each circle. We then evaluate inductive transfer on graphs with different node counts (N = 50 and N = 200) without any fine-tuning or retraining. This setting tests whether the learned partitioning policy can generalize across graph sizes. The results 
presented in Table \ref{tab:zeroshot_transformed} demonstrate that our model 
exhibits significant generalizability and outperforms the non-neural traditional baselines even if they are applied directly on these graphs. These results indicate that \ourmodel trained on graphs of a specific size can be effectively applied to graphs of different sizes. 
\subsection{Ringness and Wedgeness Evaluation}
Ringness and Wedgeness are quantitative indicators (which measure how well the generated partitions align with concentric and radial structural patterns) used to evaluate 
the proximity of a partition to Ring and Wedge partitions, 
from ring and wedge viewpoint respectively.
Ringness and Wedgeness are defined as normalized scores comparing the generated partitions to ideal concentric-ring and radial-wedge layouts.
\ifextendedversion
Detailed definitions are provided in Sec. \ref{app:xness}.
\fi
Table \ref{tab:ringness-wedgeness} presents the quantification results 
of Ringness and Wedgeness across City Traffic Graphs.  Our \ourmodel consistently achieves higher scores across both metrics, indicating that its partitions more closely follow the ring- and wedge-like structures motivating the proposed framework. In contrast, baseline methods that rely on unconstrained or node-level operations exhibit substantially lower scores, suggesting that optimizing cut objectives alone is insufficient to recover global geometric structure. These results highlight the importance of embedding domain knowledge directly into the action space.
\subsection{Inference Time and Scalability}
Table \ref{tab:inference} shows inference time of our method and competitors on City Traffic graphs with 4 partitions.
We can observe that METIS and ClusterNet have relatively low and stable running time; Spectral Clustering, while also having a shorter running time in the experiments, exhibits a rapid increase based on node number. The recent neural baseline NeuroCUT takes a longer time and shows a significant growth as the number of points increases. Our method \ourmodel takes relatively longer than METIS and ClusterNet, but is significantly faster than NeuroCUT. 
\textbf{Scalability.} To scale to larger graphs, we perform two different optimizations focusing on both the input and the model. \\
\textit{(i) Pruning:} We provide a pruning method which decreases the node number without hurting the performance of \ourmodel.
\ifextendedversion
Details are in Appendix \ref{app:shorten}.
\fi
\\ \textit{(ii) Linear-time backbone + NC summary.} We replace the quadratic Transformer backbone with a Mamba state-space model~\cite{gu2024mamba} (linear time/memory  over ordered nodes). Since the normalized cut (NC) attention required an $O(n^2)$ NC matrix, we remove it and inject a compact per-node NC summary (mean, max, top-$k$) computed in $O(n)$ without materializing dense matrices. The training reward is still the exact NC, computed from edges and the partition labels (no NC matrix). We train with the same two-stage schedule as the original method (Stage~1 and Stage~2 PPO, $2\times10^5$ episodes each), and then evaluate on larger graphs.
Table~\ref{tab:1000_nodes_scaling} reports results on 1000-node graphs as we vary the number of evaluation graphs. Across evaluation sizes, \ourmodel consistently achieves lower NC than NeuroCut (better partition quality), while remaining within the same order of magnitude of inference time (slower than NeuroCut but still in the millisecond range). This indicates our scalability changes preserve quality while keeping inference practical at scale.
\begin{table}[t]
\footnotesize
\centering
\caption{Performance on 1000-node graphs varying the number of sample graphs. Best values per row (lower is better) are in bold. Our method \ourmodel outperforms the learning-based baseline NeuroCut while still being efficient in running time.}
\label{tab:1000_nodes_scaling}
\resizebox{\columnwidth}{!}{
\begin{tabular}{c|cc|cc}
\toprule
\multirow{2}{*}{$N$} &
\multicolumn{2}{c|}{NeuroCut} &
\multicolumn{2}{c}{RIDGECUT} \\
\cline{2-5}
& NC & Time (s) & NC & Time (s) \\
\hline
200  & $4.2394 \pm 6.8576$ & $\mathbf{0.0033 \pm 0.0009}$ & $\mathbf{1.5949 \pm 3.4943}$ & $0.0085 \pm 0.0107$ \\
300  & $4.7746 \pm 7.6209$ & $\mathbf{0.0031 \pm 0.0008}$ & $\mathbf{1.4851 \pm 3.0580}$ & $0.0077 \pm 0.0091$ \\
400  & $4.8585 \pm 7.7001$ & $\mathbf{0.0032 \pm 0.0009}$ & $\mathbf{1.3480 \pm 2.7157}$ & $0.0085 \pm 0.0105$ \\
1000 & $5.4930 \pm 8.2696$ & $\mathbf{0.0033 \pm 0.0009}$ & $\mathbf{1.2563 \pm 2.1537}$ & $0.0078 \pm 0.0087$ \\
\bottomrule
\end{tabular}
}
\end{table}
\subsection{Ablation Studies}
We ablate \ourmodel and its variants.
Specifically:
\textbf{\ourmodel} is the standard Wedge-Ring Partition with two-stage training.
\textbf{\ourmodel$_{e2e}$} directly learns Wedge-Ring Partition without two-stage training.
\textbf{\ourmodel$_{sr}$} uses the same reward function during two training stages.
\textbf{\ourmodel$_{o}$} does not perform post refinement after ring-wedge partition is generated.
\textbf{\ourmodel$_{nfw}$} does not freeze the wedge action network during the second training stage.
Table \ref{tab:overall-ab} presents an ablation study that isolates the contribution of key components in \ourmodel. Removing either the ring-based or wedge-based partitioning primitives leads to a consistent degradation in performance, indicating that both components are necessary to capture complementary structural properties of the graph. Variants that relax the structured action space or replace it with more generic alternatives further exhibit noticeable drops, highlighting the importance of constraining the policy to domain-informed operations. In addition, ablations that modify the training strategy show that increased flexibility does not necessarily translate to better performance, and in some cases hinders learning due to inefficient exploration. Overall, these results demonstrate that \ourmodel's effectiveness arises from the combined design of structured actions and controlled optimization, rather than any single component in isolation.
\begin{table*}[t]
\centering
\caption{Ablation Studies: \textbf{\ourmodel$_{e2e}$} directly learns Wedge-Ring Partition without two-stage training.
\textbf{\ourmodel$_{sr}$} uses the same reward function during two training stages. \textbf{\ourmodel$_{o}$} does not perform post refinement after ring-wedge partition is generated.
\textbf{\ourmodel$_{nfw}$} does not freeze the wedge action network during the second training stage. Performance comparison on Predefined-weight, Random-weight, and City Traffic Graphs (Normalized Cut). Lower values indicate better performance. }

\label{tab:overall-ab}
\begin{tabular}{@{}l*{12}{c}@{}}
\toprule
\multirow{3}{*}{Method} & \multicolumn{4}{c}{Predefined-weight} & \multicolumn{4}{c}{Random-weight} & \multicolumn{4}{c}{City Traffic} \\
\cmidrule(lr){2-5} \cmidrule(lr){6-9} \cmidrule(l){10-13}
& \multicolumn{2}{c}{4 Part.} & \multicolumn{2}{c}{6 Part.} & \multicolumn{2}{c}{4 Part.} & \multicolumn{2}{c}{6 Part.} & \multicolumn{2}{c}{4 Part.} & \multicolumn{2}{c}{6 Part.} \\
\cmidrule(lr){2-3} \cmidrule(lr){4-5} \cmidrule(lr){6-7} \cmidrule(lr){8-9} \cmidrule(lr){10-11} \cmidrule(l){12-13}
& 50 & 100 & 50 & 100 & 50 & 100 & 50 & 100 & 50 & 100 & 50 & 100 \\
\midrule
Metis 
& .069 & .036 & .097 & .053 & .065 & .049 & .094 & .049 & .245 & .162 & .383 & .304 \\
Spec. Clust. 
& .065 & .036 & .099 & .053 & .079 & .053 & .101 & .053 & .384 & .218 & .652 & .843 \\
Bruteforce 
& .070 & .036 & .106 & .054 & .070 & .054 & .107 & .054 & .361 & .237 & .615 & .457 \\
Random 
& .076 & .040 & .144 & .074 & .080 & .072 & .142 & .072 & .209 & .095 & .512 & .341 \\
NeuroCut 
& .059 & .032 & .086 & .046 & .064 & .033 & .093 & .049 & .192 & .078 & .348 & .226 \\
ClusterNet 
& .078 & .043 & .106 & .070 & .093 & .043 & .120 & .083 & .507 & .261 & .837 & .747 \\
\hline
\ourmodel$_{sr}$ 
& .063 & .276 & .065 & .032 & .159 & .044 & .091 & .046 & .646 & .473 & .792 & .612 \\
\ourmodel$_{e2e}$ 
& .105 & .053 & .123 & .063 & .112 & .055 & .131 & .069 & .683 & .478 & .783 & .678 \\
\ourmodel$_o$ 
& \textbf{.042} & \textbf{.021} & \textbf{.062} & \textbf{.032} 
& \textbf{.057} & \textbf{.029} & \textbf{.081} & \textbf{.041} 
& .209 & .071 & .419 & .271 \\
\ourmodel$_{nfw}$ 
& .046 & .023 & .065 & .033 
& \textbf{.057} & \textbf{.029} & .082 & \textbf{.041} 
& .175 & .060 & .328 & .187 \\
\ourmodel 
& \textbf{.042} & \textbf{.021} & \textbf{.062} & \textbf{.032} 
& \textbf{.057} & \textbf{.029} & \textbf{.081} & \textbf{.041} 
& \textbf{.174} & \textbf{.060} & \textbf{.317} & \textbf{.182} \\
\bottomrule
\end{tabular}
\end{table*}
\ifextendedversion
\textit{Additional ablation studies are provided in Appendix \ref{app:ablation}}.
\fi

\section{Conclusion}
We propose \ourmodel, an RL-based framework for graph partitioning that incorporates domain-specific constraints. More specifically, \ourmodel is designed for solving the graph partitioning problem in spatial networks where \textit{good} partitions are naturally shaped as  approximate rings and wedges. We achieve this by first encoding rings and wedges using a sequential structure and use transformer-based embeddings to capture relevant spatial context. The RL model is learned using Proximal Policy Optimization where the reward is the non-differentiable Normalized Cut metric. We demonstrate that \ourmodel generalizes across graph instances and produces more meaningful and higher-quality partitions. One interesting future direction might be to generalize this framework to dynamic graphs.

\begin{acks}
This work was supported by the National Natural Science Foundation of China under Grant No. 62172107 and the National Key Research and Development Program of China under Grant No. 2018YFB\allowbreak0505000.
\end{acks}


\FloatBarrier
\bibliographystyle{ACM-Reference-Format}
\bibliography{sections/qarsumo}

\ifextendedversion
\clearpage
\appendix
\newcommand{\appsection}[1]{
  \refstepcounter{section}
  \section*{\thesection. #1}
}
\appsection{Definitions of Ringness and Wedgeness}
\label{app:xness}
We propose the Ringness and Wedgeness to evaluate whether a partition is close to the ring shape or wedge shape. 
We expect a typical Ring and Wedge partition will have the highest Ringness and Wedgeness. 
For partition $p_i \in P$, 
we define the partition ranges as $pr_i=\{\min(\boldsymbol{r}_i), \max(\boldsymbol{r}_i)\}$, 
partition angle $pa_i=\{\min(\boldsymbol{a}_i), \max(\boldsymbol{a}_i)\}$,
where $\boldsymbol{r}_i$ and $\boldsymbol{a}_i$ are polar coordinates of nodes that belongs to $p_i$.
Then we define Ringness for a partition $R_P(r) = |\{r \in pr_i\}|$, which means that how many partitions cover the radius $r$. 
For a pure ring partition, as different partitions will never overlap within their radius, $R_P(r)$ will be always 1; and for a Ring and Wedge partition, $R_P(r)$ is always 1, except the out-most wedge part, is the wedge partition number $k_w$.
For Wedgeness, we define $W_P(r) = \sum_{r \in pr_i} |pa_i|$, where $|pa_i|$ is the angle range of $pa_i$. 
For radius $r$, we only consider the partition that covers the selected range, and we sum up the angles covered by these partitions. 
The angle should equal or greater than $2\pi$, as the graph is fully partitioned by $P$. 
If a partition is a pure wedge partition, for any $r$, the Wedgeness should be exactly $2\pi$, because partition will never cover each other in any place. For Ring and Wedge Partition, if $r$ is in Wedge Partition part, the conclusion remains same as above; for Ring Partition part, only one partition is selected, and the Wedgeness is also $2\pi$.
To represent Ringness and Wedgeness more clearly, we calculate the quantification metrics for them based on the following formula:
{\scriptsize\begin{align}
\boldsymbol{W'} &= \min_{0 \le k \le \max(\boldsymbol{r})} \left( \int_{i=0}^k W_P(i) + \int_{i=k}^{\max(\boldsymbol{r})} \left( \max(W) - W_P(i) \right) \right)\label{eq:wp1}\end{align}}
\begin{align}\boldsymbol{W}_P &= \frac{Z(P) - \boldsymbol{W'}}{Z(P)} \label{eq:wp}\end{align}
\begin{align}
\boldsymbol{R}_P &= \frac{2\pi}{\max_r R_P(r)} \label{eq:rp} \\
f(x) &= \left\{\begin{array}{lcl}
1 & \text{if} & 0\le x \le k\\
\max(W) & \text{if} & k < x \le max(\boldsymbol{r})
\end{array}\right.
\end{align}
Here $Z(P)=0.5\max(\boldsymbol{r})\cdot\max(W)$ is the normalization factor.
We use a piecewise function $f$ to approximate $W_P$, and provide $\boldsymbol{W}_P$ based on the difference between $W_P$ and $f$.
For $\boldsymbol{R}_P$, we select the maximum of $R_p$.
Both $\boldsymbol{W}_P$ and $\boldsymbol{R}_P$ is scaled to $[0, 1]$, and the higher means the better.

\appsection{Theoretical Analysis: Proofs of Cheeger Bounds}
\label{app:bound}
In the graph partitioning context there exists bounds on the Cheeger constant in terms of the normalized Laplacian eigenvalues, see for example \cite{chung1997spectral} for bisection and \cite{trevisan2014highcheeger} for more general k-partitions. Intuitively, the Cheeger constant measures the size of the minimal ``bottleneck`` of a graph and it is related to the optimal partition. Since we consider a subset of all the possible partition classes, namely ring and wedge, we show that the normalized cut defined in \eqref{eq: ncut_definition} satisfies bounds similar to the classical case in the case of unweighted spider web graphs. Despite being a simpler class of graphs, these bounds give a theoretical justification of the normalized cut definition \eqref{eq: ncut_definition} and the ring-wedge shaped partition.
\subsection{Proof of Proposition\ref{prop: 1}}
\label{app:proof}
In this section we will provide all the details of the proof of Proposition\ref{prop: 1}. First we recall some background definitions and results. 
 Let $G=(V, E) $ be an undirected graph with $|V|=n$. Let $D$ be the diagonal matrix with the node degrees on the diagonal and let $L=D^{-\frac{1}{2}}(D-A)D^{-\frac{1}{2}}$ be the normalized Laplacian of $G$\footnote{For the sake of simplicity, often it will be called just Laplacian.}, where $A$ is the adjacency matrix of $G$. The matrix $L$ is positive semi-definite with eigenvalues
\begin{equation}
0 = \lambda_1 \leq \lambda_2 \leq \ldots \leq \lambda_n.
\end{equation}
For a subset $S\subseteq V$ define
\begin{equation}
\phi_G(S) = \frac{Cut(S, S^c)}{Volume(S)}
\end{equation}
and, for $1\leq k\leq n$, we define the \textit{Cheeger constants} 
\begin{equation}
\rho_G(k) = \min_{\substack{S_1,\ldots, S_k \\ \text{partition of V}}}\max_{1\leq i\leq k}\phi_G(S_i).
\end{equation}
It is known that, for $k=2$, the following inequalities hold \cite{chung1997spectral}
\begin{equation}
\frac{\lambda_2}{2} \leq \rho_G(2) \leq \sqrt{2\lambda_2}.
\end{equation} 
Analogous inequalities were proved in \cite{trevisan2014highcheeger} for every $1\leq k\leq n$
\begin{equation}
\frac{\lambda_k}{2}\leq \rho_G(k) \leq \mathcal{O}(k^2)\sqrt{\lambda_k}.
\end{equation}
If $G$ is planar, then the right-side inequality can be improved and reads
\begin{equation}
\rho_G(k) \leq \mathcal{O}(\sqrt{\lambda_{2k}}).
\end{equation}
Now let $G_{N,r}=(V,E)$ be an undirected spider web graph with $r$ rings and $N$ points for each ring. This is exactly the cartesian product of a circle graph and a path graph with $N$ and $r$ vertices respectively. Note that the ``center`` is not included in this type of graphs. Define the following custom Cheeger constants:
\begin{align}
\varphi_{N, r}(k) & = \min_{\substack{S_1,\ldots, S_k \\ \text{wedge partition of V}}}\max_{1\leq i\leq k}\phi_{G_{N,r}}(S_i)\\
\psi_{N, r}(k) & = \min_{\substack{S_1,\ldots, S_k \\ \text{ring partition of V}}}\max_{1\leq i\leq k}\phi_{G_{N,r}}(S_i).
\end{align}
For an illustration of ring and wedge partitions see Figure \ref{fig:ring_wedge_partitions}.
\begin{figure*}
    \centering
    \includegraphics[scale=0.7]{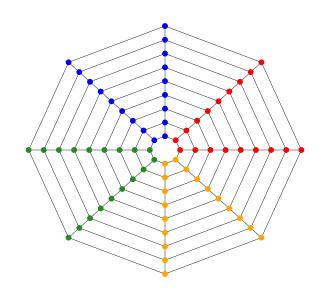}
    \hspace{1cm}\includegraphics[scale=0.7]{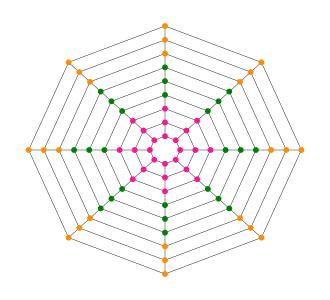}
    \caption{Examples of $k=4$ wedge (left) and $k=3$ ring (right) partitions.}
    \Description{Two spider-web graph examples showing wedge partitions and ring partitions.}
    \label{fig:ring_wedge_partitions}
\end{figure*}
From now on we will assume $G=G_{N,r}$ to be a spider web graph with $r$ rings and $N$ points for each ring. We can compute bounds on $\varphi_{N, r}(k)$ and $\psi_{N, r}(k)$.
\begin{lemma}\label{lem: ringwedge_bounds}
Given a spider-web graph $G_{N,r}$ the ring and wedge Cheeger constants can be bounded as follows
\begin{equation}
\varphi_{N, r}(k) \leq \frac{r}{\lfloor\frac{N}{k}\rfloor (2r-1)}, \quad \psi_{N, r}(k) \leq \frac{1}{2\lfloor\frac{r}{k}\rfloor}.
\end{equation}
\end{lemma}
\begin{proof}
The strategy will be to choose suited partitions for which it is possible to compute explicitly the cuts and the volumes. We start from the wedge Cheeger constant $\varphi_{N,r}(k)$. Given a wedge partition $S_1, \ldots, S_k$ each subset $S_i$ has cut exactly $2r$. Moreover, we assume that the $S_i$'s are maximally symmetric, meaning that the $S_i$'s all have $\lfloor \frac{N}{k}\rfloor$ or $\lfloor\frac{N}{k}\rfloor+1$ nodes in each ring. These observation read
\begin{equation}
\begin{split}
\varphi_{N, r}(k) \leq & \max_{1\leq i\leq k}\frac{2r}{Volume(S_i)} \\ 
 = & \frac{2r}{\min_{1\leq i\leq k}Volume(S_i)}.
\end{split}
\end{equation}
The wedge subset with minimum volume is given by  one that has $\lfloor{\frac{N}{k}}\rfloor$ nodes in each ring, hence
\begin{equation}
\begin{split}
\varphi_{N, r}(k) \leq & \frac{2r}{\underbrace{4}_{\substack{\text{degree of}\\ \text{inner ring} \\ \text{nodes}}}\underbrace{(r-2)}_{\substack{\text{number of}\\ \text{inner rings}}}\underbrace{\lfloor{\frac{N}{k}}\rfloor }_{\substack{\text{number of}\\ \text{points in} \\ \text{each ring}}}+ \underbrace{3}_{\substack{\text{degree of} \\ \text{outer ring} \\ \text{nodes}}} \underbrace{2}_{\substack{\text{number of} \\ \text{outer rings}}} \underbrace{\lfloor{\frac{N}{k}}\rfloor }_{\substack{\text{number of}\\ \text{points in} \\ \text{each ring}}}} \\
= & \frac{r}{\lfloor\frac{N}{k}\rfloor(2r-1)}.
\end{split}
\end{equation}
For the ring Cheeger constant the setting is more complicated since different subsets might have different cut, in contrast with the case of wedge partitions. Given a ring partition $S_1, \ldots, S_k$ which is maximally symmetric, i.e., all the $S_i$'s have $\lfloor\frac{r}{k}\rfloor$ or $\lfloor\frac{r}{k}\rfloor +1$ nodes in each ring, we order the $S_i$'s from the center to the outermost ring. Note that $S_1$ has some nodes with degree $3$ while $S_2$ has all nodes with degree $4$ for $k>2$. For $k$> 2, we consider two cases:
\begin{itemize}
\item $k$ divides $r$. In this case we only need to compare $S_1$ and $S_2$. It holds that
\begin{align}
    \phi_G(S_1) = & \frac{N}{4N(\frac{r}{k}-1)+3N} = \frac{1}{4\frac{r}{k}-1} \\ 
    \phi_G(S_2) = & \frac{2N}{4N\frac{r}{k}} = \frac{1}{2\frac{r}{k}},
\end{align}
since $S_1$ and $S_2$ have cut $N$ and $2N$ respectively. Thus, $\phi_G(S_2)\geq \phi_G(S_2)$ which implies $\psi_{N,r}(k)\leq \frac{1}{2\frac{r}{k}}$.
\item $k$ does not divide $r$. In this case, we assume $S_1$ has $\lfloor\frac{r}{k}\rfloor +1$ nodes and $S_2$ has $\lfloor\frac{r}{k}\rfloor$ nodes. Then
\begin{align}
    \phi_G(S_1) = & \frac{N}{4N\lfloor\frac{r}{k}\rfloor+3N} = \frac{1}{4\lfloor\frac{r}{k}\rfloor+3} \\ 
    \phi_G(S_2) = & \frac{2N}{4N\lfloor\frac{r}{k}\rfloor} = \frac{1}{2\lfloor\frac{r}{k}\rfloor}.
\end{align}
Thus, $\phi_G(S_2)\geq \phi_G(S_2)$ which implies $\psi_{N,r}(k)\leq \frac{1}{2\lfloor\frac{r}{k}\rfloor}$.
\end{itemize}
For $k=2$, if $k$ divides $r$, then $\phi_G(S_1)=\phi_G(S_2)=\frac{1}{4\frac{r}{k}\rfloor-1}\leq \frac{1}{2\frac{r}{k}}$. If $k$ does not divide $r$, then if $S_1$ has $\lfloor\frac{r}{k}\rfloor +1$ nodes and $S_2$ has $\lfloor\frac{r}{k}\rfloor$ nodes, we have $\phi_G(S_1)= \frac{1}{4\lfloor\frac{r}{k}\rfloor+3}$ and $\phi_G(S_2) = \frac{1}{4\lfloor\frac{r}{k}\rfloor-1} \leq \frac{1}{2\lfloor\frac{r}{k}\rfloor}$.
Putting together the above inequalities, we get that $\psi_{N,r}(k)\leq\frac{1}{2\lfloor\frac{r}{k}\rfloor}$.\end{proof}
In the next sections we will provide the proof details for the two bounds in Proposition \ref{prop: 1}. We start from the case of wedge partitions. 
\subsubsection{Wedge partitions}
We will prove the bound on wedge Cheeger constants in terms of the eigenvalues of the circle graph with $N$ vertices $C_N$. We will consider only the case of $k>1$ since the first eigenvalues is always $0$ and spider web graphs are connected. First we recall that the eigenvalues of $C_N$ are
\begin{equation}
1-\cos\left(\frac{2\pi k}{N}\right), \quad 0\leq k\leq N-1,
\end{equation}
see \cite{chung1997spectral}. In particular, we have the following result.
\begin{lemma}\label{lem:1}
Let $C_N$ be the circle graph with $N$ vertices. Then the $k$-th eigenvalues of the normalized Laplacian of $C_N$ is given by
\begin{equation}
\lambda^C_k = 1 - \cos\left( \frac{ 2 \pi \lfloor\frac{k}{2}\rfloor }{N} \right), \quad 1\leq k\leq N.
\end{equation} 
\end{lemma}
\begin{proof} If we order the values of $\left\{1-\cos\left(\frac{2\pi (k-1)}{N}\right)\right\}_{k=1}^{N}$ we notice that
\begin{equation}\label{eq: lambda_sep}
\lambda^C_k = \left\{\begin{array}{lcl}
f(\frac{k}{2}) & \text{if} & k\in 2\mathbb{Z} \\
f(\frac{k-1}{2}) & \text{if} & k\notin 2\mathbb{Z}
\end{array}\right.
\end{equation}
where $f(k) = 1-\cos\left(\frac{2\pi k}{N}\right)$. Writing together the two pieces in \eqref{eq: lambda_sep} we get
\begin{equation}
\lambda^C_k = 1 - \cos\left( \frac{ 2 \pi \lfloor\frac{k}{2}\rfloor }{N} \right),\quad  1\leq k\leq N.
\end{equation} 
\end{proof}
Now we will prove some inequalities that together will build the final wedge Cheeger inequality.
\begin{lemma}\label{lem:2}
$\pi\lfloor\frac{k}{2}\rfloor\frac{1}{N} \leq \frac{\pi}{2}$, for $2\leq k \leq N$.
\end{lemma}
\begin{proof} 
Since $k\leq N$ we have the following inequality
\begin{equation}
\pi\lfloor\frac{k}{2}\rfloor\frac{1}{N} \leq \pi\lfloor\frac{N}{2}\rfloor\frac{1}{N}\left\{\begin{array}{lcl}
= \frac{\pi}{2} & \text{if} & k \in 2\mathbb{Z} \\
= \pi\frac{N-1}{2}\frac{1}{N}\leq \frac{\pi}{2} & \text{if} & k \notin 2\mathbb{Z} \\
\end{array}\right.
\end{equation}
\end{proof}
\begin{lemma}\label{lem:3}
$2\lfloor\frac{k}{2}\rfloor \geq \frac{k}{2}$, for $2\leq k\leq N$.
\end{lemma}
\begin{proof}
 If $k$ is even then $2\lfloor\frac{k}{2}\rfloor =2 \frac{k}{2}\geq \frac{k}{2}$. If $k$ is odd, then $2\lfloor\frac{k}{2}\rfloor = 2\frac{k-1}{2} =k-1\geq \frac{k}{2}$ for $2\leq k\leq N$.
\end{proof}
\begin{lemma}\label{prop:1}
$\sqrt{\lambda^C_{k}}\geq \frac{\sqrt{2}}{4}\frac{1}{\lfloor\frac{N}{k}\rfloor}$, for $2\leq k\leq N$.
\end{lemma}
\begin{proof}
It holds that
\begin{align}
\sqrt{\lambda^C_k} = &\sqrt{1 - \cos\left( \frac{ 2 \pi \lfloor\frac{k}{2}\rfloor }{N} \right)}  \label{eq:inequality1}\\ 
=& \sqrt{2}\sin\left( \frac{ \pi \lfloor\frac{k}{2}\rfloor }{N} \right) \label{eq:inequality2}\\
\geq & \sqrt{2}\frac{2}{\pi}\left( \frac{ \pi \lfloor\frac{k}{2}\rfloor }{N} \right) \label{eq:inequality3}\\
= & 2\sqrt{2} \lfloor\frac{k}{2}\rfloor \frac{1}{N} \\
\geq & \sqrt{2} \frac{k}{2}\frac{1}{N}\label{eq:inequality4} \\ 
\geq & \frac{\sqrt{2}}{2} \frac{1}{\lfloor\frac{N}{k}\rfloor +1} \\
\geq & \frac{\sqrt{2}}{2} \frac{1}{2\lfloor\frac{N}{k}\rfloor} \\
= & \frac{\sqrt{2}}{4}  \frac{1}{\lfloor\frac{N}{k}\rfloor}
\end{align}
where \eqref{eq:inequality1} follows from Lemma \ref{lem:1}, \eqref{eq:inequality2} follows from the fact that $\cos(2x)= 1-2\sin^2(x)$, \eqref{eq:inequality3} follows from the fact that $\frac{\sin(x)}{x} > \frac{2}{\pi}$ for $x\in\left[-\frac{\pi}{2}, \frac{\pi}{2}\right]$ and from Lemma \ref{lem:2}, \eqref{eq:inequality4} follows from Lemma \ref{lem:3}.
\end{proof}
Combining the results in Lemma \ref{prop:1} together with the ones in Lemma \ref{lem: ringwedge_bounds} we get the following result.
\begin{prop}
For a spider web graph $G_{N,r}$ we have $\varphi_{N, r}(k) \leq \frac{2r}{2r-1}\sqrt{2\lambda^C_{k}}$, for $2\leq k\leq N$.
\end{prop}
\subsubsection{Ring partitions}
Similarly as for wedge partitions, we will prove a bound on the ring Cheeger constants in terms of the eigenvalues of the path graph with $r$ vertices $P_r$. Some of the computations are analogous to the ones in the previous section, so we will skip the details for these. 
We recall that the eigenvalues of $P_r$ are
\begin{equation}
\lambda^P_k = 1-\cos\left(\frac{\pi (k-1)}{r-1}\right), \quad 1\leq k\leq r,
\end{equation}
see \cite{chung1997spectral}. We have the following inequality for the ring Cheeger constant.
\begin{lemma}\label{lem: ring_ineq}
$  \sqrt{\lambda^P_{k}}\geq \frac{\sqrt{2}}{2}\frac{1}{2\lfloor\frac{r}{k}\rfloor}$, for $2\leq k\leq r$.
\end{lemma}
\begin{proof}
It holds that
\begin{align}
\sqrt{\lambda^P_{k}} = &\sqrt{1 - \cos\left( \frac{ \pi (k-1)}{2(r-1)} \right)} \\ 
= & \sqrt{2}\sin\left( \frac{ \pi (k-1) }{2(r-1)} \right) \\
\geq & \sqrt{2}\frac{2}{\pi}\left(  \frac{ \pi (k-1) }{2(r-1)} \right)\\
\geq & \frac{\sqrt{2}}{2}\frac{k}{r-1}\\
= & \frac{\sqrt{2}}{2}\frac{1}{\frac{r}{k}-\frac{1}{k}} \\
\geq & \frac{\sqrt{2}}{2}\frac{1}{\lfloor\frac{r}{k}\rfloor +1 -\frac{1}{k} }\\
\geq & \frac{\sqrt{2}}{2}\frac{1}{2\lfloor\frac{r}{k}\rfloor }\label{eq: ring1}\\
\end{align}
where the inequality \eqref{eq: ring1} follows from the fact that 
\begin{equation}
   \lfloor\frac{r}{k}\rfloor +1 -\frac{1}{k} \leq 2\lfloor\frac{r}{k}\rfloor. 
\end{equation}
\end{proof}
Combining the results in Lemma \ref{prop:1} together with the ones in Lemma \ref{lem: ring_ineq} we get the following result.
\begin{prop}
For a spider web graph $G_{N,r}$ we have $\psi_{N, r}(k) \leq \sqrt{2\lambda^P_{k}}$, for $2\leq k\leq N$.
\end{prop}

\appsection{The Ring Wedge Transformer Module}
\label{app:transformer}
\begin{figure*}[t]
  \centering
  \includegraphics[width=1\linewidth]{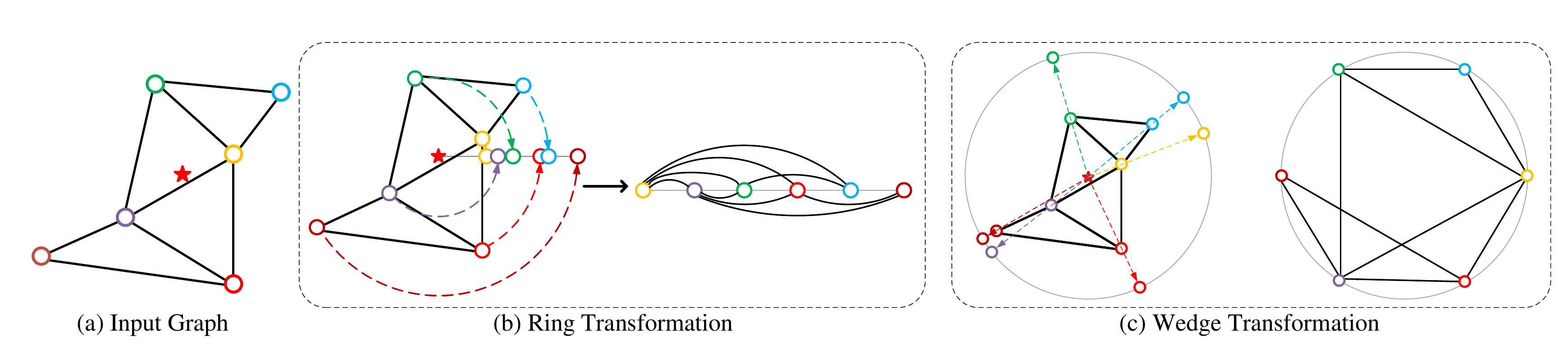}
  \caption{Example of Wedge Transform and Ring Transform. In Wedge Transform, nodes are projected to a circle, then the difference of angles of adjacent nodes are adjusted to the same. In Ring Transform, nodes are projected to a line. The edge connections and their weights are not changed in both transformation.}
  \Description{Diagram showing wedge transformation to a circle and ring transformation to a line while preserving edge connections and weights.}
  \label{fig:transform}
\end{figure*}
\ourmodel utilizes a Transformer backbone to leverage information from transformed graphs, enabling it to handle varying node counts and enhancing its scalability for diverse applications without the need for fine-tuning after training. The Transformer architecture is illustrated in Figure \ref{fig:model} (b).
The transformer processes inputs from the Pre-Calculation module, specifically Cut Weight and Volume Matrix, along with the Current Partition from the input graph. These are fed into \(n\) Transformer blocks, yielding node embeddings from the final hidden state.
To effectively manage Current Partition, we represent each node's selection status Partition Selection with a 0-1 array, then it is combined with Cut Weight and transformed through a linear layer to generate hidden states, which are subsequently augmented with positional embeddings. 
In the transformer module, we introduce Partition Aware Multi-Head Attention (PAMHA) to replace the original Multi-Head Attention (MHA) layer. PAMHA incorporates the Volume Matrix and Current Partition into its attention mask. 
An element-wise transformation on \(V\) produces an attention mask of shape \(N \times N\) for PAMHA, allowing the model to learn the significance of different nodes.
For Current Partition, we observe that partitions splitting between nodes \(i\) and \(i+1\) do not affect the normalized cut calculations on the right of \(i+1\). For instance, in the circular graph with six nodes depicted in Figure \ref{fig:model}, 
two wedge partition angles are already established. When partitions are made 
among the red, green, blue, and yellow nodes, these modifications have no impact 
on the contributions from the brown and purple nodes. 
Consequently, we create an attention mask focusing solely on the effective range of nodes.
Finally, the transformer module outputs node embeddings, which are then input to the PPO module. 

\appsection{Training and Testing of \ourmodel}
\label{app:traintest}
In this section, we describe the training and testing strategies in details.
\subsection{Training Strategy}
In previous model designs, \ourmodel has effectively extracted information from graphs. 
However in RL, the initial strategies are inherently randomized, 
presenting a significant challenge in learning an optimal strategy. 
Specifically, the ring partition and wedge partition can interfere with 
one another. For instance, if the ring partition consistently selects the 
smallest radius as its action, the wedge partition will struggle to learn 
any valid policy. This occurs because in such situation, the Normalized Cut 
for the ring partition is considerably large, and the final Normalized Cut 
is significantly influenced by the ring partition, regardless of the actions 
taken by the wedge partition. Consequently, training a ring partition with 
a low-quality wedge partition strategy will also encounter similar difficulties.
To mitigate the aforementioned problem, we have divided the policy training 
into two distinct stages, as illustrated in 
Figure \ref{fig:model} (c) \ding{172} and \ding{173}. 
In the first Wedge Training stage, we implement a randomized ring selection 
method to replace ring selection strategies in \ourmodel. In this phase, 
\ourmodel is only responsible for determining and optimizing the wedge partitioning. 
To enhance the model's concentration on developing an effective wedge 
partition strategy, we intentionally ignored the Normalized Cut of ring 
partitions when evaluating the reward received by \ourmodel. This approach ensures 
that the model is concentrated on mastering a robust wedge partitioning strategy.
In the second Ring Training stage, we allow \ourmodel to determine both the ring 
and wedge partitions. However, we have observed that permitting the model 
to adjust its parameters related to wedge partitioning can lead to a 
decline in its ability to effectively execute wedge partitioning before 
it has fully developed an optimal strategy for ring partitioning. 
To mitigate this issue, we freeze the parameters of the wedge partitioning 
modules in \ourmodel, because \ourmodel has already established a robust wedge partitioning 
strategy across various radii.
The only exception is the Critic Projection Header. In the previous stage, 
we modified this header to solely utilize the Normalized Cut of wedge partitions 
as the reward, which is inconsistent with current reward definition. 
Hence, during the Ring Training stage, both Critic Projection Headers are re-initialized and trained.
In PPO, since the strategies are only determined by actor model, allowing 
the critic to be trainable does not interfere with the learned policy.
\subsection{Testing Strategy}
After \ourmodel is fully trained, it can generate partitions directly during the 
Partition Generation stage. Initially, the model will perform a ring 
partition, followed sequentially by wedge partitions, as illustrated in 
Figure \ref{fig:model} (c) \ding{174}.
While we have proved that ring and wedge partitions possess similar upper bounds 
under certain constraints, in real-world graphs, these partitions may not yield 
optimal results when outliers exist. As illustrated in 
Figure \ref{fig:model} (c) \ding{175}, the original grouping of two nodes 
are reversed during the execution of a pure ring and wedge partition. 
To mitigate this problem, we propose a Post Refinement Stage, where
nodes within the same partition that are not directly connected will be 
separated into distinct partitions. We then adopt a greedy approach to select 
the partition with the highest Normalized Cut and  subsequently merge it 
with adjacent partitions. This post-refinement technique effectively reduces 
the number of outlier nodes and enhances the quality of the resulting partitions produced by our method.
Finally, since PPO is a policy-gradient based method that generates an action 
probability distribution, relying on a single segmentation may not directly 
lead to the optimal solution. Therefore, we can perform multiple random 
samplings to obtain different partitions and select the best of them.

\appsection{Statistics and Hyper Parameters}
\label{app:stat}
We show the statistics of datasets and hyper parameters of model in Table \ref{tab:stas}.
\begin{table*}[ht]
\small
\centering
\caption{
Statistics of datasets and hyper parameters of model.
}
\label{tab:spider_graph_specs}
\begin{tabular}{@{}clll@{}}
\toprule
\textbf{Type} & \textbf{Parameter} & \textbf{Values} & \textbf{Description} \\ \midrule
\multirow{5}{*}{\rotatebox{90}{\centering Synthetic}} & Nodes & \{50, 100\} & Nodes on each circle \\
&Circles & same as Nodes & Number of concentric rings \\
&Low Weight & \{2, 4, 6\} & Intra-partition edge weights \\
&High Weight & \{10, 15, 20\} & Inter-partition edge weights \\
&Random Weight & Uniform(1, 10) & Edge weights for random \\
\hline
\multirow{2}{*}{\rotatebox{90}{\centering Real}} & Nodes & \{50, 100\} & Nodes on each graph \\
& Edge Weight & [1, 372732] & Edge weights\\
\hline
\multirow{9}{*}{\rotatebox{90}{\centering \makecell{Hyper Parameters}}} &Partitions & \{4,6\} & Number of partitions \\
& Hidden Size & 64 & Hidden size of Transformer \\
& Layer Number & 3 & Transformer layer number \\
& Learning Rate & 1e-3 & Learning rate \\
& Batch Size & 256 & Batch size \\
& Discount Factor & 0.9 & Discount factor in RL \\
& Training Step & 400,000 & Steps for training \\
& Test Number & 100 & Test graph number \\
& Sample Number & 10 & Sampled partition number \\
\bottomrule
\end{tabular}
\label{tab:stas}
\end{table*}
\appsection{Ablation Studies}
\label{app:ablation}
\begin{table*}[t]
\centering
\caption{Transfer performance measured by Normalized Cut. Methods that do not support transfer or unable to perform results are excluded. Models are trained on 100 nodes and tested on 50 or 200 nodes.  }

\label{tab:zeroshot-ab}
\begin{tabular}{@{}l*{12}{c}@{}}
\toprule
& \multicolumn{4}{c}{Predefined-weight} & \multicolumn{4}{c}{Random-weight} & \multicolumn{4}{c}{City Traffic} \\
\cmidrule(lr){2-5} \cmidrule(lr){6-9} \cmidrule(l){10-13}
Partition & \multicolumn{2}{c}{4 Part.} & \multicolumn{2}{c}{6 Part.} & \multicolumn{2}{c}{4 Part.} & \multicolumn{2}{c}{6 Part.} & \multicolumn{2}{c}{4 Part.} & \multicolumn{2}{c}{6 Part.} \\
\cmidrule(lr){2-3} \cmidrule(lr){4-5} \cmidrule(lr){6-7} \cmidrule(lr){8-9} \cmidrule(lr){10-11} \cmidrule(l){12-13}
Nodes & 50 & 200 & 50 & 200 & 50 & 200 & 50 & 200 & 50 & 200 & 50 & 200 \\
\midrule
METIS &      {.069}  & .019  & {.097}  & {.027}  & {.065}  & {.016}  & {.094}  & {.024}  & .245  & {.048}  & {.383}  & {.086} \\
Bruteforce & .070  & {.018}  & .106  & .028  & .070  & .018  & .107  & .028  & .361  & .175  & .615  & .311 \\
Random &     .076  & .021  & .144  & .037  & .080  & .021  & .142  & .037  & {.209}  & .512  & .512  & .212 \\
\hline
\ourmodel$_{sr}$ & {.219} & {.201} & {.068} & {.018} & {.085} & {.023} & {.092} & {.024} & {.664} & {.305} & {.831} & {.224} \\
\ourmodel$_{e2e}$ & .103 & .027 & .107 & .027 & .107 & .028 & .123 & .033 & .645 & .327 & .863 & .442 \\
\ourmodel$_o$ & {.052} & {.013} & {.066} & {.016} & {.061} & {.016} & {.087} & {.022} & {.182} & {.031} & {.472} & {.014} \\
\ourmodel$_{nfw}$ & {.053} & {.013} & {.066} & {.017} & {.061} & {.016} & {.087} & {.104} & {.150} & {.023} & {.327} & {.090} \\
\ourmodel & {.052} & {.013} & {.066} & {.017} & {.061} & {.016} & {.087} & {.022} & {.158} & {.023} & {.323} & {.085}
\color{black} \\
\bottomrule
\end{tabular}
\end{table*}
\begin{table}[t]
\centering
\caption{Performance comparison on City Traffic Graphs (Normalized cut). Lower values indicate better performance. }

\label{tab:overall-ab-new}
\begin{tabular}{@{}l*{4}{c}@{}}
\toprule
\multirow{3}{*}{Method} & \multicolumn{4}{c}{City Traffic} \\
\cmidrule(lr){2-5} 
& \multicolumn{2}{c}{4 Part.} & \multicolumn{2}{c}{6 Part.} \\
\cmidrule(lr){2-3} \cmidrule(lr){4-5} 
& 50 & 100 & 50 & 100 \\
\midrule
GPS & 0.696 & 0.521 & 0.821 & 0.711 \\
Transformer & 0.186 & 0.122 & 0.360 & 0.225 \\
\ourmodel w/o PAMHA & 0.169 & 0.075 & 0.374 & 0.193 \\
\ourmodel & {0.174} & {0.060} & {0.317} & {0.182}
\color{black} \\
\bottomrule
\end{tabular}
\end{table}
We propose ablation studies of \ourmodel and its variants.
They are as follows:
\textbf{\ourmodel} is the standard Wedge-Ring Partition with two-stage training.
\textbf{\ourmodel$_{e2e}$} directly learns Wedge-Ring Partition without two-stage training.
\textbf{\ourmodel$_{sr}$} uses the same reward function during two training stages.
\textbf{\ourmodel$_{o}$} does not perform post refinement after ring-wedge partition is generated.
\textbf{\ourmodel$_{nfw}$} does not freeze the wedge action network during the second training stage.
We give additional ablation studies in the following to show the effectiveness of proposed methods.
The performance of ablation models are shown in Table \ref{tab:zeroshot-ab}.
\subsection{Multi-stage Training and Testing}
In Section \ref{subsec:pipeline}, we propose multi-stage traininig and testing strategies. In training, we 
propose to train the wedge partition model firstly, and randomly select ring partitions. The radius are uniformly selected from 0 to 80\% of maximum radius.
After wedge partition model is trained, we re-initialize the critic projection header of wedge model, and fix the other parts of wedge model to train the ring model part.
We show the performance without two stage training as \ourmodel$_{e2e}$.
From Table \ref{tab:overall-ab} and Table \ref{tab:zeroshot-ab}, we can find that without two stage training, the model is not able to converge, because 
a bad policy of either ring or wedge will affect the learning process of each other, and make the model hard to converge.
\begin{figure*}[ht]
  \centering
  \captionsetup[subfigure]{labelformat=empty} 
  \begin{subfigure}{.45\textwidth}
    \centering
    \includegraphics[width=\textwidth]{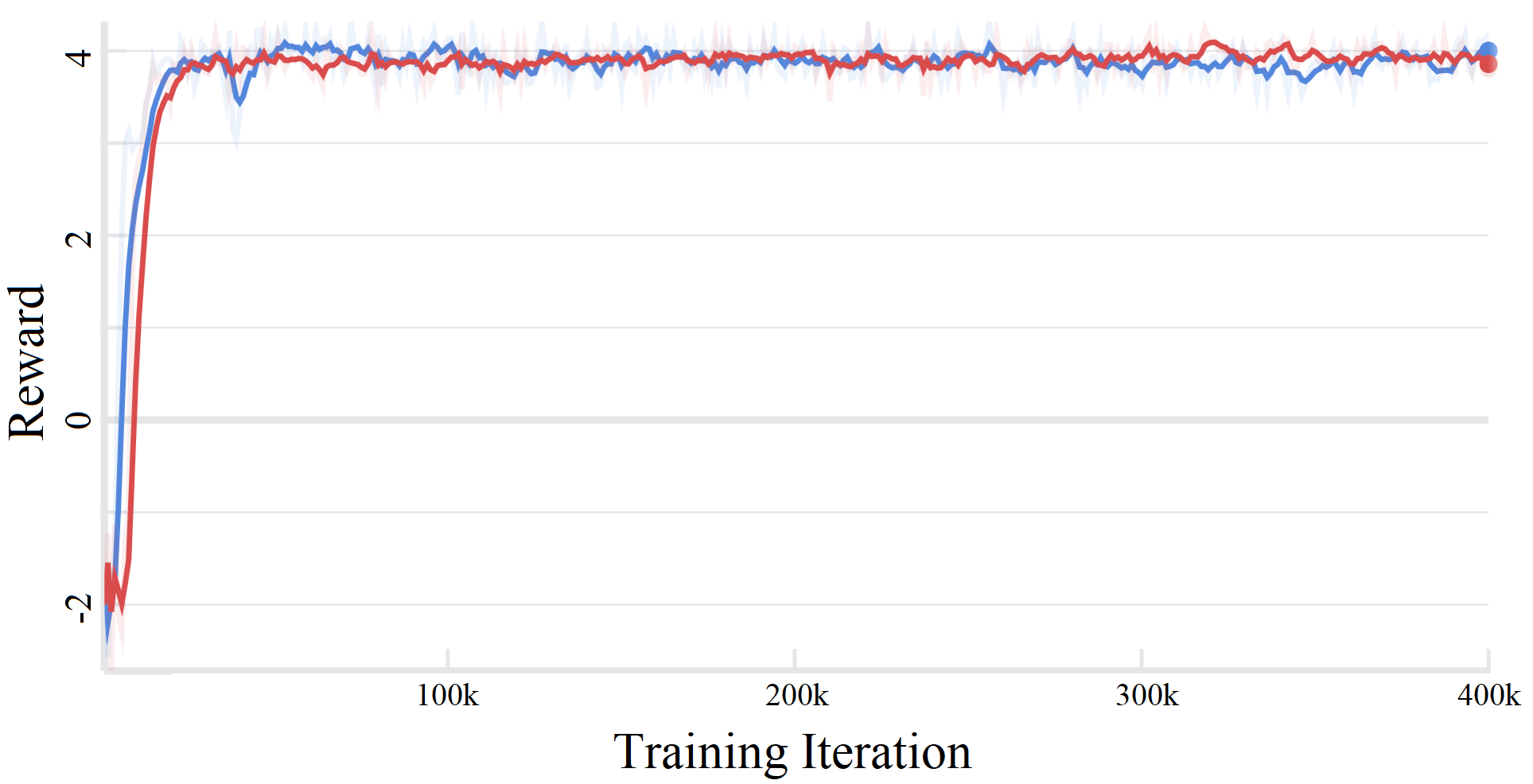}
    \caption{Training reward curve of \ourmodel}
  \end{subfigure}
  \hfill 
  \begin{subfigure}{.45\textwidth}
    \centering
    \includegraphics[width=\textwidth]{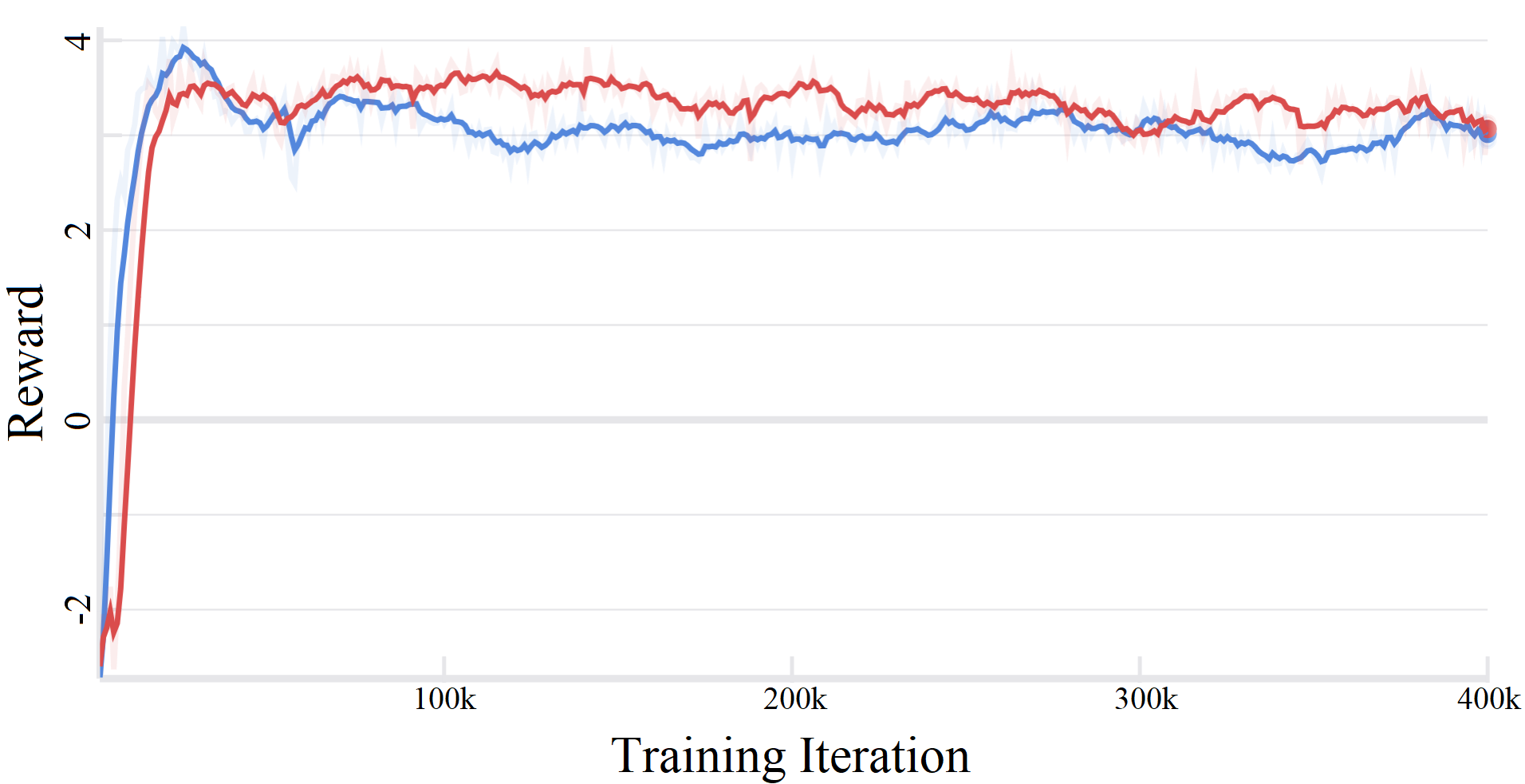}
    \caption{Training reward curve of \ourmodel$_{nfw}$}
  \end{subfigure}
   
  \begin{subfigure}{.45\textwidth}
    \centering
    \includegraphics[width=\textwidth]{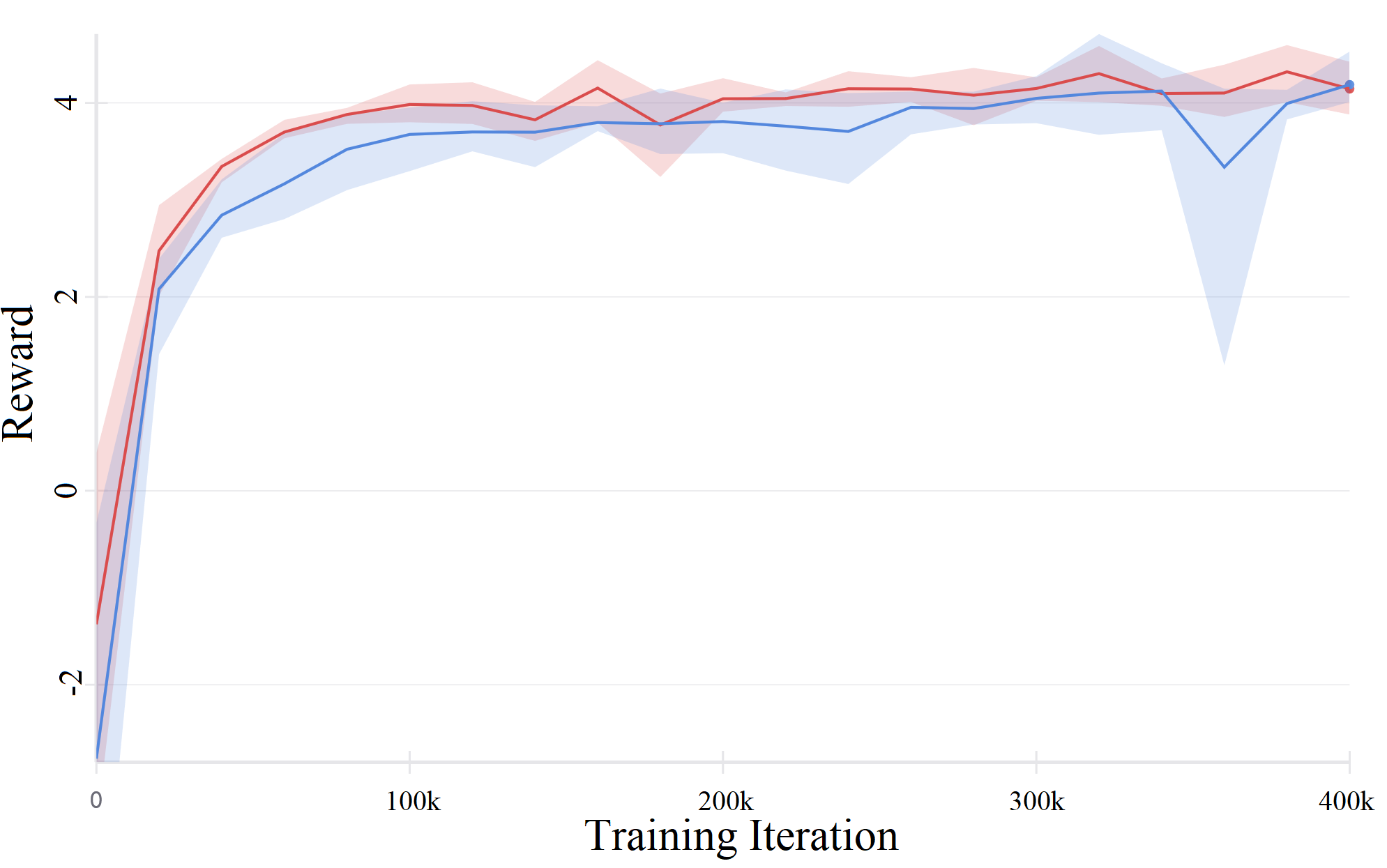}
    \caption{Testing reward curve of \ourmodel}
  \end{subfigure}
  \hfill
  \begin{subfigure}{.45\textwidth}
    \centering
    \includegraphics[width=\textwidth]{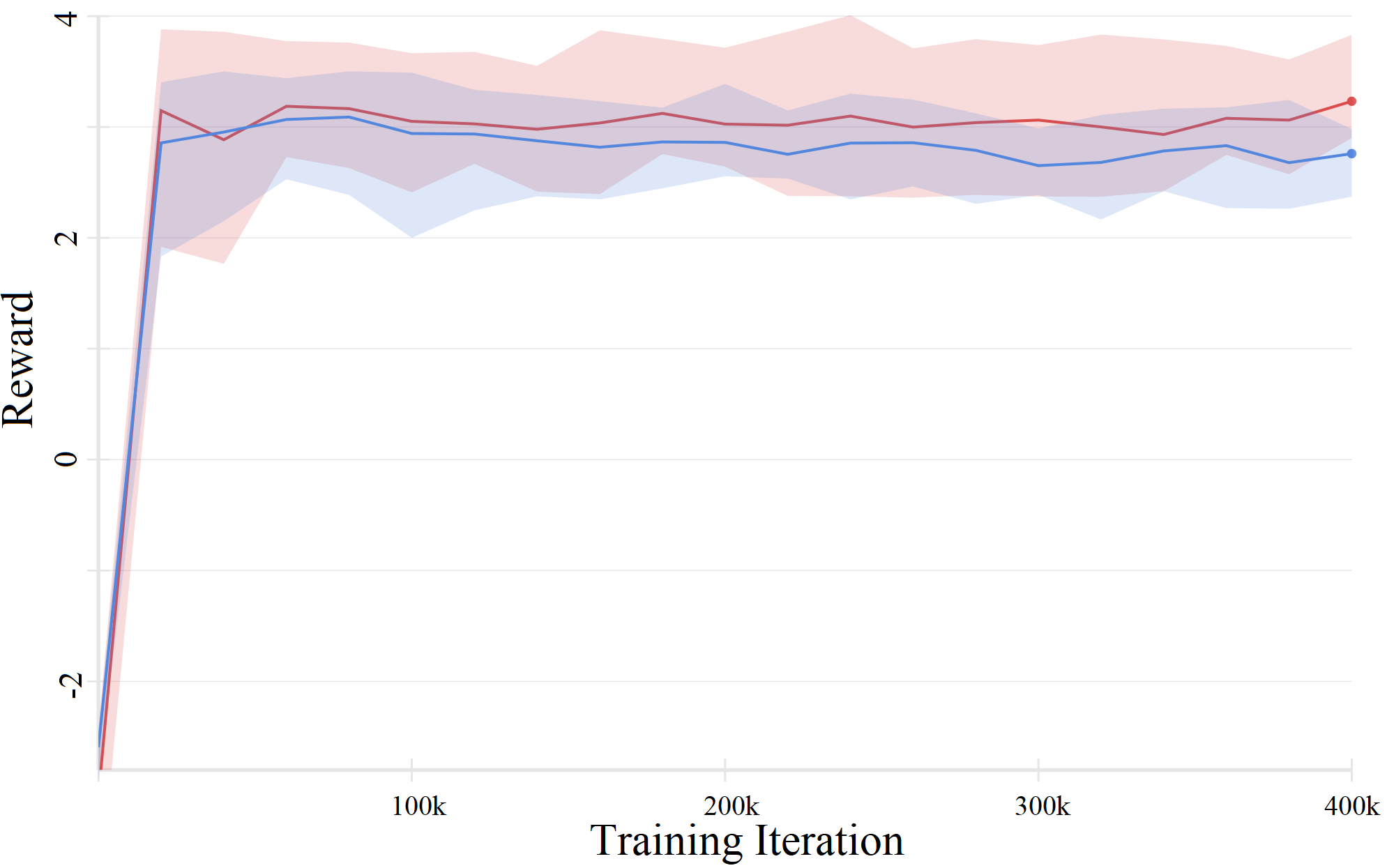}
    \caption{Testing reward curve of \ourmodel$_{nfw}$}
  \end{subfigure}
	  \caption{Reward curves during training and testing of Predefined-weight Graph. Red is with 50 node number and blue is with 100 node number. We individually perform 4 tests for each checkpoint, using the curve to represent the average test results, while the shaded area indicates the maximum and minimum values observed during the tests.}
  \Description{Training and testing reward curves for RidgeCut and RidgeCut without frozen wedge training on predefined-weight graphs.}
	  \label{fig:reward-curve}
\end{figure*}
\subsection{Different Reward Function in Two-stage}
As mentioned in Section \ref{subsec:pipeline}, in training wedge partition, we change the reward function from global Normalized Cut to the Normalized Cut that only considering wedge partitions.
This avoids the impact of poor ring partition selection, as ring partition is performed by a random policy, and may give poor partitions.
For example, if the random policy selects a very small radius, the normalized cut of circle partitions will be very big, which makes the reward received from different wedge partition identical.
We show the performance when the reward function keeps same, i.e. always considering the normalized cut of ring partition in two stage training, as \ourmodel$_{sr}$. In Table \ref{tab:overall-ab} and \ref{tab:zeroshot-ab}, we can find that their performance is worse than \ourmodel, because their wedge partitioning policies are not strong enough.
As reward function will change in two-stage training, we will re-initialize the critic net of wedge model in the second stage, as mentioned before.
\subsection{Fix Wedge Partition Policy}
In the second training stage, we fix the wedge model to avoid changing the policy. \ourmodel$_{nfw}$ shows the performance when wedge partition policy is not fixed. We can find that the performance decreases if wedge partition policy is not fixed, and leads to bad policy in several test cases. 
This is because if we allow the action net change, it may forget learned policy before a valid policy has learned by ring partition, and leads to worse performance and instability during the training.
The reward curve during training and testing, which is shown in Figure \ref{fig:reward-curve}, also supports the conclusion. 
It has been observed that not fixing the action net results in lower and more unstable rewards for the model during training. Moreover, the performance during testing tends to become more variable and does not show further improvements as training progresses.
\subsection{Post Refinement}
We perform post refinement after performing the ring and wedge partition,
which splits existing partition result by the connectivity of nodes, then reconstruct new partitions by combining the partition which has biggest Normalized Cut with its adjacent partitions.
As ring and wedge partitions on synthetic graphs are always connected, this post refinement will not change the performance of \ourmodel on synthetic dataset.
However, in real dataset, sometimes the graph shape is not compatible to ring and wedge partition, and the results may not good enough. With post refinement, we can further decrease the Normalized Cut on such situation. In Table \ref{tab:overall-ab}, we show the performance improvements with post refinement on real dataset, the normalized cut is decreased 22.4\% on average.
\subsection{Graph Center Selection}
\begin{figure*}[tb]
    \centering
	    \includegraphics[width=0.4\textwidth]{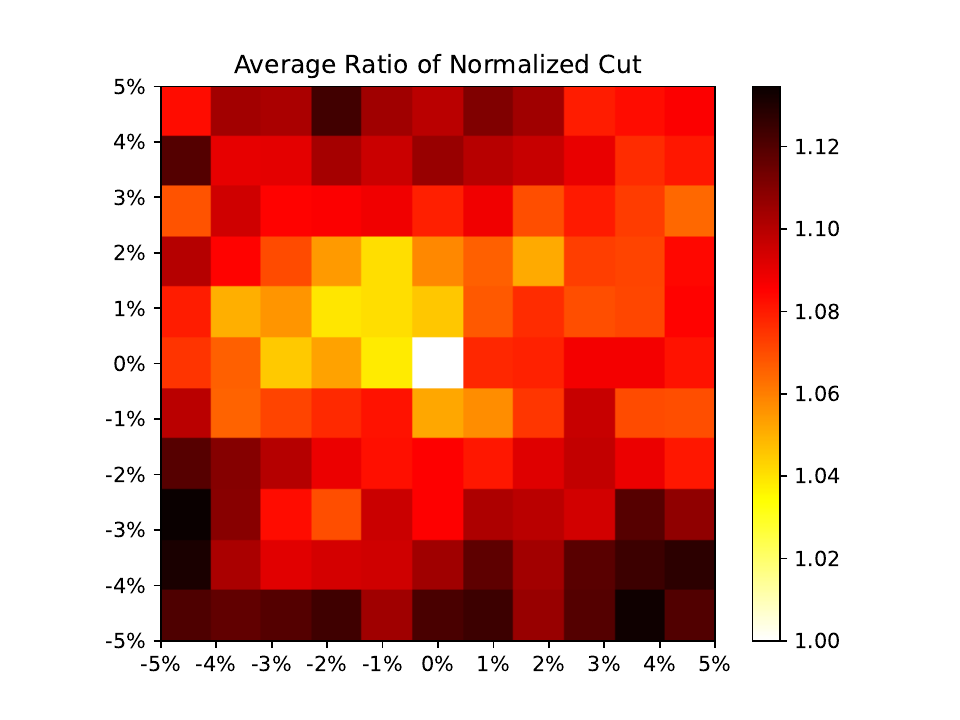}
	    \hspace{0.5cm}\includegraphics[width=0.4\textwidth]{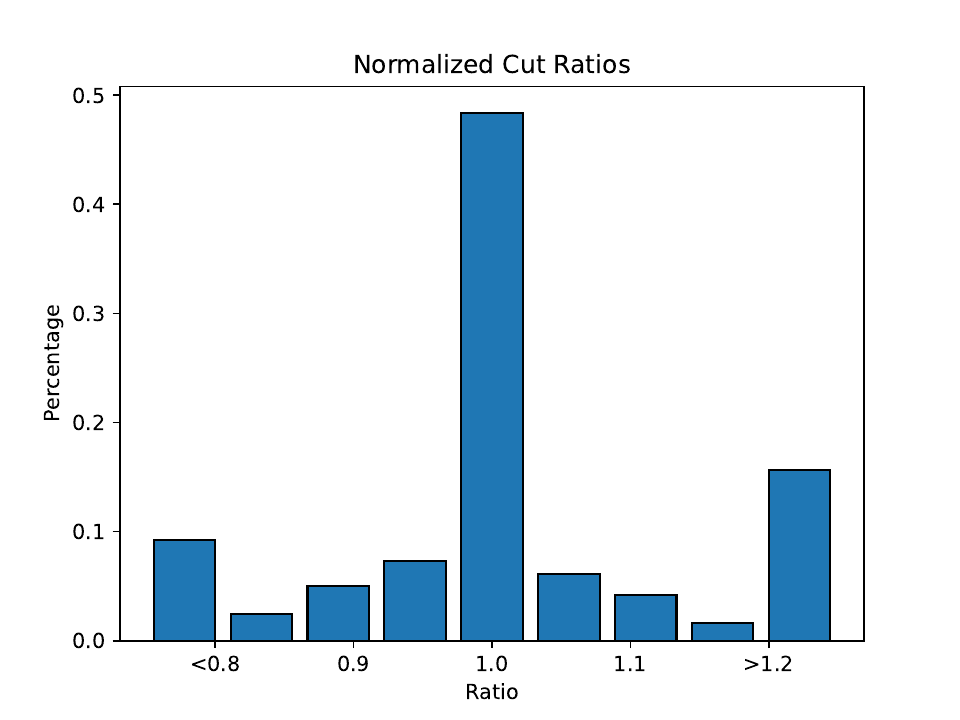}
	    \caption{Heatmap and histogram of Normalized Cut when applying offsets to the center. Values are normalized by Normalized Cut of unoffset center. Lower value is better.}
	    \Description{Heatmap and histogram summarizing normalized cut ratios under different center offsets.}
	    \label{fig:center_ratio}
\end{figure*}
We conducted experiments on the test set of City Traffic graphs with 50 nodes, which contains 100 graphs. Based on the maximum aspect ratio of the graphs, we offset the centroid by a distance of up to 5\% and recalculated the results of Normalized Cut. For better comparison, we normalized the results using the Normalized Cut from the unoffset scenario. A normalized value closer to zero indicates better performance, with a value of 1 signifying that the results are the same as in the unoffset case.
Figure \ref{fig:center_ratio} (left) illustrates the results for various offsets from the centroid. We observe that any offset from the centroid results in a worse performance, and with greater offsets correlating to a more significant decline.
In Figure \ref{fig:center_ratio} (right), we present the histogram of results across all the aforementioned offsets and graphs. We find that in nearly half of the cases where offsets were applied, the resulting errors remained within 5\%. Furthermore, applying offsets tends to lead to worse outcomes more frequently. Thus, in this paper, we opted to use the centroid as the center of the graph. Figure \ref{fig:center_ratio} (right) also shows that in approximately 15\% of cases, offsetting the centroid yielded improvements of over 10\%. In the future, we can propose a more effective strategy for centroid selection to enhance the algorithm's performance.
\subsection{Effectiveness of Graph Transformation, \ourmodel and PAMHA}
We demonstrate the effectiveness of our proposed Graph Transformation methods, \ourmodel and PAMHA, as presented in Table \ref{tab:overall-ab-new}. The methods evaluated are as follows:
\begin{itemize}
\item GPS \cite{DBLP:conf/nips/RampasekGDLWB22}: It integrates GNNs and Transformers to effectively capture both local and global information. It serves as the benchmark for evaluating the performance of GNNs and Transformers in addressing the Normalized Cut problem using Ring and Wedge Partition techniques.
\item Transformer: We implement Ring and Wedge Transformations on the graph, utilizing node coordinates as positional embeddings and edge weights as attention masks, followed by the application of a standard Transformer model.
\item \ourmodel w/o PAMHA: This configuration maintains the same architecture as \ourmodel but substitutes PAMHA with conventional Multi-Head Attention.
\end{itemize}
From the results, we observe that GPS struggles to produce meaningful results in City Traffic. In contrast, the Transformer model demonstrates significantly improved outcomes when compared to GPS. This finding reinforces the notion that constraining the action space of RL agents based on domain knowledge is essential for enhancing performance. Next, we integrated the \ourmodel structure, which facilitates data pre-calculation, and observed an increase in performance. This indicates that the pre-calculation module is beneficial in addressing the problem. Finally, by incorporating PAMHA into the Transformer framework, we achieved the full \ourmodel architecture, which exhibited superior performance relative to other variants.
\subsection{Performance of Other Methods}
\begin{table}[t]
\centering
\caption{Performance comparison on City Traffic Graphs (Normalized cut). Lower values indicate better performance. }

\label{tab:overall-comp}
\begin{tabular}{@{}l*{4}{c}@{}}
\toprule
\multirow{3}{*}{Method} & \multicolumn{4}{c}{City Traffic} \\
\cmidrule(lr){2-5} 
& \multicolumn{2}{c}{4 Part.} & \multicolumn{2}{c}{6 Part.} \\
\cmidrule(lr){2-3} \cmidrule(lr){4-5} 
& 50 & 100 & 50 & 100 \\
\midrule
DMon & 0.998 & 1.000 & 1.000 & 1.000 \\
MinCutPool & 0.549 & 0.365 & 0.864 & 0.634 \\
Ortho & 0.924 & 0.892 & 0.997 & 0.982 \\
\ourmodel & {0.174} & {0.060} & {0.317} & {0.182}
\color{black} \\
\bottomrule
\end{tabular}
\end{table}
\begin{table}[t]
\small
\centering
\caption{Comparison of Hidden State and Learning Rate for Different Methods.}

\label{tab:hyp-comp}
\begin{tabular}{@{}lcc@{}}
\toprule
         & {Hidden State} & {Learning Rate} \\ \midrule
Search Range          & 16, 32, 64, 128, 256  & 1e-4, 3e-4, 1e-3, 3e-3, 1e-2 \\ 
Dmon                  & 128                   & 1e-4                    \\ 
MinCutPool            & 32                   & 1e-3                    \\ 
Ortho                 & 256                   & 1e-2                    \\ 
\bottomrule
\end{tabular}
\end{table}
We incorporate additional comparative methods Analyzed on City Traffic Graph, which include:
\begin{itemize}
\item DMon \cite{DBLP:journals/jmlr/TsitsulinPPM23}: A neural attributed graph clustering method designed to effectively handle complex graph structures.
\item MinCutPool \cite{DBLP:conf/icml/BianchiGA20}: This method focuses on optimizing the normalized cut criterion while incorporating an orthogonality regularizer to mitigate unbalanced clustering outcomes.
\item Ortho \cite{DBLP:journals/jmlr/TsitsulinPPM23}: This refers to the orthogonality regularizer that is utilized in both DMon and MinCutPool, ensuring greater balance in the clustering process.
\end{itemize}
All models were trained using the same settings as \ourmodel. We conducted a grid search on the hyperparameters of the three aforementioned methods and selected the optimal combination of hyperparameters to train on other datasets. The search range and the selected hyperparameters are detailed in Table \ref{tab:hyp-comp}.
The results are summarized in Table \ref{tab:overall-comp}. From the findings, it is evident that Dmon fails to effectively learn the partition strategy, resulting in most outcomes being invalid (with Normalized Cut values of 1). Ortho performs slightly better but still tends to yield unbalanced results. In contrast, MinCutPool demonstrates a significant improvement over the previous methods; however, it still exhibits a considerable range compared to our proposed \ourmodel.

\appsection{Detail of the Model Pipeline}
\label{app:pipeline}
We detail the model transformation formulas below. Since the processes for ring and wedge partitioning are similar, we focus on the wedge partitioning pipeline and note the differences later.
Let \(G\) be the input graph and \(P\) the current partition. We apply Wedge Transformation to \(G\) to obtain a linear graph \(G'\). Define \(n_i\) as the \(i\)-th node on this line, with \(n\) candidate actions. Action \(a_i\) corresponds to selecting the radius of the ring partition between \(n_i\) and \(n_{i+1}\). The input embedding is constructed as follows:
\[
\boldsymbol{X}_i = \text{Linear}_{CW}(\text{Cut}_i \oplus \text{PS}_i) + \text{Pos}_{i/|N|} \to \mathbb{R}^d,
\]
where \(\text{Linear}_{CW}\) is a linear transformation, \(\text{Cut}_i\) is the Cut Weight between \(n_i\) and \(n_{i+1}\), \(\text{PS}_i\) is the Partition Selection for \(n_i\), and \(\text{Pos}_{i/|N|}\) is the positional embedding scaled based on the total number of nodes.
In \ourmodel, we derive the attention masks \(M^P\) and \(M^V\) as follows:
\begin{equation}
\begin{aligned}
M^P_{i,j} &= 
\begin{cases}
0 & \text{if } i \text{ and } j \text{ are in the same partition} \\
-\infty & \text{otherwise}
\end{cases}, \\
M^V_{i,j} &= \text{Linear}_V(V_{i,j}).
\end{aligned}
\end{equation}
When nodes \(i\) and \(j\) are in different partitions, the attention weight is set to \(-\infty\) to prevent their influence on each other. Denote \(H^0_i = X_i\); the \(t\)-th hidden states \(H^t_i\) are computed as follows:
\begin{align}
\boldsymbol{Q}^t, \boldsymbol{K}^t, \boldsymbol{V}^t &= \text{Linear}_{\text{Att}}(\boldsymbol{H}^t), \\
\boldsymbol{O}^t &= \text{Softmax}(\boldsymbol{QK}^t + \boldsymbol{M^V} + \boldsymbol{M^P}).
\end{align}
The output for each node at the \(t\)-th layer is:
\begin{align}
Y^t_i &= \text{LayerNorm}\left(\sum_{k=1}^{N} O^t_{i,k} V_k / \sqrt{d} + H^t_i\right), \\
H^{t+1}_i &= \text{LayerNorm}(Y^t_i + \text{FFN}(Y^t_i)).
\end{align}
The Transformer has \(T\) layers, with \(\boldsymbol{E} = \boldsymbol{H}^T\) as the output embeddings.
In the PPO module, we calculate action probabilities and predicted values using \(\boldsymbol{E}\):
\begin{align}
\text{Logit}_i &= \text{Linear}_A(E_i) \to \mathbb{R}, \\
\textbf{Prob} &= \text{Softmax}(\text{Logit}), \\
\text{PredictedValue} &= \text{Linear}_{PV}(\text{Attention}(\boldsymbol{E})) \to \mathbb{R}.
\end{align}
We sample actions from the probability distribution and train the model using rewards and predicted values.
For ring partitioning, two key differences arise. First, we employ Ring Transformation on the graph. Second, the positional embedding is 2-dimensional, reflecting the adjacency of the first and last nodes. Specifically, for a circular node with coordinates \((x, y)\), we use:
\[
PE = \text{Linear}_{2D}(x \oplus y) \to \mathbb{R}^d
\]
to generate the positional embedding.

\appsection{Pruning Input Sequence Length for \ourmodel}
\label{app:shorten}
To improve the capability of processing larger graphs, we propose a pruning 
strategy on existing graphs when the transformed sequences are input to the \ourmodel.
Specifically, after Ring/Wedge Transformation, we can choose partial
of nodes as the valid ones, and ignore other nodes to be input to 
the \ourmodel. Currently we select nodes that with Minimum Cut Weights.
The results are shown in the following, we apply it on 500-node
graphs, the Normalized Cut has no loss when we decrease the node
number into half, and it only increases around 25\% when we only
keep 100 nodes (see Table ~\ref{tab:nc-pruning}).
\begin{table}[h]
\centering
\caption{Normalized Cut (NC) under Node Pruning with Respect to 500 Nodes}
\label{tab:nc-pruning}
\begin{tabular}{c|c|c}
\hline
\textbf{Node Count} & \textbf{Normalized Cut (NC)} & \textbf{$\Delta$ (\%)} \\
\hline
500 & 0.007205 & 0.00 \\
350 & 0.007088 & -1.63 \\
250 & 0.006927 & -3.86 \\
150 & 0.009035 & +25.38 \\
100 & 0.009090 & +26.15 \\
50  & 0.009899 & +37.44\\
20  & 0.010406 & +44.51 \\
\hline
\end{tabular}
\end{table}
\FloatBarrier
\appsection{Dynamic Programming Algorithm for Ring Partition Phase}
\label{app:dp}
We show the pesudo-code of dynamic programming algorithm used in Ring Partition phase in Algorithm \ref{algo:dp}. This allows us performing ring partition only once. The time complexity of this algorithm is $O(n^2k)$.
The DP matrix  $dp[i,j]$ stores the minimum normalized cut value when partitioning the first  $i$ nodes into $j$ segments, with transitions recorded in the predecessor matrix  $pre[i][j]$. The final loop traces back from the last segment's optimal value to reconstruct the partition indices by following $pre$ entries iteratively.
\begin{algorithm}
\caption{Dynamic Programming for Ring Partition} 
\label{algo:dp}
\KwIn{Precomputed cut weight matrix $Cut$, volume matrix $Volume$, number of partitions $k$} 
\KwOut{Optimal Normalized Cut $res$, partition indices $P$}
$sector\_nc[i, j] \gets (Cut[i] + Cut[j]) / Volume[i, j]$ for all $i, j$;
// $dp[i, j]$ means the best result when we perform partition on node $i$ and it is the $j$-th partition \\
$dp[i,j] \gets \infty$ for all $i,j$; 
$dp[0,0] \gets 0$; 
// $pre[i,j]$ records where the value for $dp[i,j]$ transits from \\
$pre[i,j] \gets 0$;
\For{$i$ from $1$ to $|Cut| - 1$}{ 
    \For{$j$ from $1$ to $k - 1$}{ 
        // enumerate all $p<i$ and assume last partition is from $p$ to $i$ \\
        \For{$p$ from $1$ to $i - 1$}{ 
            $agg\_res[p] \gets \max(dp[p, j - 1], sector\_nc[p, i])$; \\
        }
        $pre[i,j] \gets \arg\min(agg\_res)$; \\
        $dp[i,j] \gets agg\_res[argmin]$; \\
    } 
}
// The last partition should be from $p$ to $|Cut|$, update it to $dp[p, k - 1]$ \\
\For{$p$ from $1$ to $|Cut| - 1$}{ 
    $dp[p, k - 1] = \max(dp[p, k - 1], sector\_nc[p, |Cut| - 1])$;
}
$result \gets dp[res\_x,res\_y]$; 
// get final partition indices \\
$r_x \gets \arg\min(dp[:, k - 1])$; 
$r_y \gets k - 1$; 
$P\gets\{\}$;
\While{$r_y>0$}{
    $P\gets P\cup \{r_x\}$;\\
    $r_x\gets pre[r_x,r_y]$;\\
    $r_y\gets r_y-1$;
}
\KwRet $result, P$
\end{algorithm}

\FloatBarrier
\appsection{Pseudo Codes of Algorithms}
\label{app:pseudo}
{\footnotesize
{\normalsize
In this section we give pseudo codes for algorithms, including Ring and Wedge Transformation, Valume and Cut calculation, Ring and Wedge Transformer, PPO and full training pipeline.
\par}
\begin{algorithm}
    \caption{Ring Transformation}
    \label{alg:graph_to_line}
    \KwIn{graph $G = (V,E,W,o)$}
    \KwOut{Converted line graph $G_l$}
    $r\gets \text{radius of } V - o$\;
\For{each element $i$ from $1$ to $|r|$ }{
    // $\text{rank of } r[i] \text{ in the sorted list of } r$ \\
    $\text{Index}[i] \gets \sum_{j=1}^{n} \mathbf{1}(r[j] \leq r[i])$\;
}
\For{each element $i$ from $1$ to $|E|$}{
    $E_{new}[i]\gets \left\{\text{Index}[E[i].x], \text{Index}[E[i].y]\right\}$\;
}
\For{each element $i$ from $1$ to $|V|$}{
    $V_{new}[i]\gets (\text{Index}[i], 0)$\;
}
    \KwRet $G_c = (V_{new}, E_{new}, W, (0, 0))$
\end{algorithm}
\begin{algorithm}
    \caption{Wedge Transformation}
    \label{alg:graph_to_circle}
    \KwIn{graph $G = (V,E,W,o)$}
    \KwOut{Converted circle graph $G_c$}
    $a\gets \text{angles of } V - o$\;
\For{each element $i$ from $1$ to $|a|$}{
    // $\text{rank of } a[i] \text{ in the sorted list of } a$ \\
    $\text{Index}[i] \gets \sum_{j=1}^{n} \mathbf{1}(a[j] \leq a[i])$\;
}
\For{each element $i$ from $1$ to $|a|$}{
    $a_{new}[i]\gets\frac{2\pi}{|a|} \text{Index}[i]$ \;
}
\For{each element $i$ from $1$ to $|E|$}{
    $E_{new}[i]\gets \left\{a_{new}[E[i].x], a_{new}[E[i].y]\right\}$\;
}
\For{each element $i$ from $1$ to $|a_{new}|$}{
    $V_{new}[i]\gets (\sin(a_{new}[i]), \cos(a_{new}[i]))$\;
}
    \KwRet $G_c = (V_{new}, E_{new}, W, (0, 0))$
\end{algorithm}
\begin{algorithm}
    \caption{Volume and Cut for Line}
    \label{alg:nc_for_line}
    \KwIn{Line graph $G_l = (V,E,W,o)$}
    \KwOut{Cut $Cut$ and Volume $Volume$}
    $a\gets \text{angles of} V - o$\;
    \For{$e, w$ in $E, W$}{
        \If{$e.x < e.y$}{
            $x, y \gets e.x, e.y$
        }
        \Else{
            $x, y \gets e.y, e.x$
        }
        \For{$i$ from $x$ to $y$}{
            $Cut[i] \gets Cut[i] + w$\;
        }
        \For{$i$ from $1$ to $x$}{
            \For{$j$ from $y$ to $|V|$}{
                $Volume[i, j] \gets Volume[i, j] + w$\;
            }
        }
    }
    \KwRet $Cut, Volume$
\end{algorithm}
\begin{algorithm}
    \caption{Volume and Cut for Circle}
    \label{alg:nc_for_circle}
    \KwIn{Circle graph $G_c = (V,E,W,o)$}
    \KwOut{Cut $Cut$ and Volume $Volume$}
    $a\gets \text{angles of} V - o$\;
    \For{$e, w$ in $E, W$}{
        $x, y \gets e.x, e.y$
        \If{$e.x > e.y$}{
            $y \gets y + |V|$\;
        }
        \For{$i$ from $x$ to $y$}{
            $Cut[i \% |V|] \gets Cut[i \% |V|] + w$\;
        }
        \For{$i$ from $1$ to $x$}{
            \For{$j$ from $y$ to $|V|$}{
                $Volume[i, j] \gets Volume[i, j] + w$\;
            }
        }
    }
    \For{$i$ from $1$ to $|V|$}{
        \For{$j$ from $1$ to $i - 1$}{
            // when $i > j$, means the direction that cross n-to-1 part \\
            $Volume[i, j] = Volume[j, i]$\;
        }
    }
    \KwRet $Cut, Volume$
\end{algorithm}
\begin{algorithm}
    \caption{Ring Wedge Transformer with Ring Partition}
    \label{alg:wrt_ring}
    \KwIn{Line graph $G_l = (V,E,W,o)$, current partition $P$}
    \KwOut{Embeddings for each nodes $emb$}
    $Cut, Volume \gets Alg4(G_c)$\;
    // shape [N, 1] to [N, H] \;
    $x\gets Linear(Cut)$\;
    // shape [N, N, 1] to [N, N, H] to [N, N, 1] \;
    $VMask \gets Linear(Volume)$ \;
    $PMask[i,j] \gets 0 \text{if i and j are in same partition}$ \;
    $PMask[i,j] \gets -\infty \text{if i and j are in different partition}$ \;
    // Pos is positional embedding, in circle partition, just same as normal NLP Transformer
    $H_0 = x + \text{Pos}$ \;
    // L is layer number \;
    \For{i from 1 to L}{
        $Q, K, V \gets \text{Linear}(H_{i-1})$ \;
        $A \gets QK^T + VMask + PMask$ \;
        $H'_i \gets \text{Norm}(AV) + H_{i-1}$ \;
        $H_i \gets \text{Norm}(\text{Linear}((H'_i)) + H'_i$ \;
    }
    \KwRet $H_L$
\end{algorithm}
\begin{algorithm}
    \caption{Ring Wedge Transformer with Wedge Partition}
    \label{alg:wrt_circle}
    \KwIn{Circle graph $G_l = (V,E,W,o)$, current partition $P$}
    \KwOut{Embeddings for each nodes $emb$}
    $Cut, Volume \gets \text{Algorithm}3(G_c)$\;
    // shape [N, 1] to [N, H] \;
    $x\gets Linear(Cut)$\;
    // shape [N, N, 1] to [N, N, H] to [N, N, 1] \;
    $VMask \gets Linear(Volume)$ \;
    $PMask[i,j] \gets 0 \text{ if i to j are in same partition}$ \;
    $PMask[i,j] \gets -\infty \text{ if i and j are in different partition}$ \;
    // Pos is positional embedding, x-y coords on the circle
    $H_0 = x + \text{Pos}$ \;
    // L is layer number \;
    \For{i from 1 to L}{
        $Q, K, V \gets \text{Linear}(H_{i-1})$ \;
        $A \gets QK^T + VMask + PMask$ \;
        $H'_i \gets \text{Norm}(AV) + H_{i-1}$ \;
        $H_i \gets \text{Norm}(\text{Linear}((H'_i)) + H'_i$ \;
    }
    \KwRet $H_L$
\end{algorithm}
\begin{algorithm}
    \caption{PPO with Embeddings}
    \label{alg:ppo}
    \KwIn{Embeddings for each nodes, i.e. actions, $emb$}
    \KwOut{Action policy logits $a$, and critic for current state $v$}
    // [N, H] to [N, H] to [N, 1] \;
    $a \gets \text{Linear(Activate(Linear(}emb)))$ \;
    // [N, H] to [1, H] to [1, 1] \;
    $c \gets \text{Linear(Activate(Attention(}emb)))$ \;
    \KwRet $a, c$
\end{algorithm}
\begin{algorithm}
    \caption{Full Training Pipeline}
    \label{alg:full}
    \KwIn{Graph $G = (V, E, W, o)$, target partition number $P_{max}$, target ring partition number $P_c$}
    \KwOut{Next partition $a$}
    $P \gets \{V\}\}$ \;
    $samples \gets \{\}$ \;
    \While{not converge}{
        // perform $P_{max}$ steps to generate partition and save into samples
        \For{i from 1 to $P_{max}$}{
            \If{$|P| <= P_c$}{
                // do ring partition
                $G_l \gets \text{GraphToLine}(G)$ \;
                $Emb \gets \text{TransformerWithRing}(G_l, P)$ \;
                $p, critic \gets \text{PPO}(Emb)$ \;
                $a \gets $ sample action from $p$ \;
                $r \gets$ radius of $G_l.V[action]$ \;
                $P' \gets$ partition p by circle with radius $r$ \;
            }
            \Else{
                // do wedge partition
                $G_c \gets \text{GraphToCircle}(G)$ \;
                $Emb \gets \text{TransformerWithWedge}(G_l, P)$ \;
                $p, critic \gets \text{PPO}(Emb)$ \;
                $a \gets $ sample action from $p$ \;
                $angle \gets$ angle of $G_l.V[action]$ \;
                $P' \gets$ partition p by wedge with angle $angle$ \;
            }
            \If{$|P| = P_{max}$}{
                $r \gets \text{NormalizedCut}(G, P)$
            }
            \Else{
                $r \gets 0$
            }
            $samples.add((G, P, p, critic, a, r))$ \;
            $P \gets P'$
        }
        // calculate loss and train with samples \;
        \If{$|samples| = target\_size$}{
            \For{$sample$ in $samples$}{
                $p_{old}, c_{old}, r \gets sample$
                // here use sample as PPO input, in fact sample will do same as above to calculate p and critic. 
                $p, critic, critic' \gets \text{PPO}(sample)$ \;
                $adv \gets r - \gamma critic' + critic $ \;
                $loss_p \gets \text{clip}(p / p_{old} * adv)$ \;
                $loss_v \gets (r - \gamma critic' + critic)^2$ \;
                $loss_{ent} \gets \text{Entropy}(p)$ \;
                $L \gets w_p loss_p + w_v loss_v + w_{ent} loss_{ent}$ \;
                Backward loss $L$ \;
            }
            $samples \gets \{\}$
        }
    } 
\end{algorithm}

}

\fi

\end{document}